\documentclass[reqno,12pt]{amsart}
\usepackage{amsfonts,amssymb,latexsym,epsfig,amsmath}
\usepackage{hyperref}
\usepackage{graphicx}
\usepackage{caption}
\usepackage{subcaption}
\usepackage{srcltx}
\usepackage{pstricks,pst-plot}

\newtheorem{lem}{Lemma}[section]
\newtheorem{theo}[lem]{Theorem}
\newtheorem{propo}[lem]{Proposition}
\newtheorem{defi}[lem]{Definition}
\newtheorem{coro}[lem]{Corollary}
\newtheorem{rema}[lem]{Remark}

\numberwithin{equation}{section}
\def\dfrac{\displaystyle\frac}

\def\R{\mathbb{R}}
\def\epsilon{\varepsilon}
\def\RR{\mathbb{R}^2_+}
\def\FR{\partial\mathbb{R}^2_+}
\def\CIN{\int_{\RR}e^{ay}|\nabla w|^2dxdy}

\def\BA{\textit{B}_a}
\def\INF{\inf_{w \in \BA}E_a(w)}
\def\RN{\mathbb{R}^n_+}

\def\Ha{H^1_a (\RR)}
\newcommand{\norm}[1]{\lVert#1\rVert}

\begin{document}

\title[Traveling wave solutions in a half-space for boundary reactions]
{Traveling wave solutions in a half-space\\
for boundary reactions}
\author[Xavier Cabr\'e, Neus C\'onsul, Jos\'e V. Mand\'e]
{Xavier Cabr\'e, Neus C\'onsul, Jos\'e V. Mand\'e\\}
\address
{X.C. Icrea and Universitat Polit\`ecnica de Catalunya\\
Departament de Matem\`atica Aplicada I\\
Diagonal 647, 08028 Barcelona, Spain} \email
{xavier.cabre@upc.edu}
\address
{N.C. Universitat Polit\`ecnica de Catalunya\\
Departament de Matem\`atica Aplicada I\\
Diagonal 647, 08028 Barcelona, Spain} \email
{neus.consul@upc.edu}
\address
{J.V.M. Universitat Polit\`ecnica de Catalunya\\
Departament de Matem\`atica Aplicada~I\\
Diagonal 647, 08028 Barcelona, Spain} \email
{josemande@gmail.com}
\thanks{The authors were supported by grants MINECO MTM2011-27739-C04-01 and GENCAT
2009SGR-345.}

\begin{abstract}
We prove the existence and uniqueness of a traveling front and of its speed 
for the homogeneous heat equation in the half-plane with a Neumann boundary reaction term of 
non-balanced bistable type or of combustion type. 
We also establish the monotonicity of the front and, in the
bistable case, its behavior at infinity. In contrast with the classical bistable interior 
reaction model, its behavior at the side of the invading state is of power type,
while at the side of the invaded state its decay is exponential. 
These decay results rely on the construction of a family of explicit bistable traveling fronts.
Our existence results are obtained via a variational method, while the uniqueness of  
the speed and of the front rely on a comparison principle and the sliding method.
\end{abstract}

\maketitle

\section{Introduction}

This paper concerns the problem
\begin{equation}
\label{parabolic}
\begin{cases}
v_t -\Delta v =0&\text{ in } \R^2_+ \times (0,\infty)\\
\displaystyle{\frac{\partial v}{\partial\nu}} =f(v) &\text{ on }
\partial\R^2_+ \times (0,\infty)
\end{cases}
\end{equation} 
for the homogeneous heat equation in a half-plane with a nonlinear Neumann
boundary condition.  To study the propagation of fronts given an initial condition, it is
important to understand first the existence and properties of traveling fronts  ---or traveling waves--- for \eqref{parabolic}. Taking $ \R^2_+=\{(x,y)\in \R^2 \, :\, x>0\}$,
these are solutions of the form $v(x,y,t):=u(x,y-ct)$ for some speed $c\in\R$. 
Thus, the pair $(c,u)$ must solve the elliptic problem 
\begin{equation}
\label{problem}
\begin{cases}
\Delta u + cu_y=0&\text{ in } \R^2_+=\{(x,y)\in \R^2 \, :\, x>0\}\\
\displaystyle{\frac{\partial u}{\partial\nu}} =f(u) &\text{ on }
\partial\R^2_+,
\end{cases}
\end{equation}
where $\partial u/\partial\nu=-u_x$ is the exterior normal derivative of~$u$ on 
$\partial\R^2_+=\{x=0\}$ , $u$ is real valued,  and $c \in \R$. 
 
We look for solutions $u$ with $0<u<1$ and having the limits 
\begin{equation}\label{limits0}
\lim_{y\to -\infty} u(0,y)=1 \quad \textrm{ and } \quad \lim_{y\to +\infty} u(0,y)=0.
\end{equation}

Our results apply to nonlinearities $f$ of non-balanced bistable type or of combustion type, 
as defined next.

\begin{defi}
\label{nonlinearterm}
Let $f$ be $C^{1,\gamma}([0,1])$, for some $\gamma\in (0,1)$, satisfy
\begin{equation}
\label{zeroes}
f(0)=f(1)=0
\end{equation}
and, for some $\delta\in (0,1/2)$,
\begin{equation}
\label{noninc}
f'\leq 0\; \mbox{ in } (0,\delta)\cup (1-\delta,1).
\end{equation}

{\rm (a)} We say that $f$ is of positively-balanced bistable type if it
satisfies \eqref{zeroes}, \eqref{noninc}, that  
 $f$ has a unique zero ---named $\alpha$--- in $(0,1)$, and that
 it is ``positively-balanced'' in the sense that
 \begin{equation}\label{posint}
 \int_0^1f(s)ds > 0.
 \end{equation}

{\rm (b)} We say that $f$ is of combustion type if it
satisfies \eqref{zeroes}, \eqref{noninc}, and that
there exists $0<\beta<1$ (called the ignition temperature)
such that $f\equiv 0$ in $(0,\beta)$ and 
\begin{equation}\label{fposbeta}
f>0 \quad\text{ in } (\beta,1). 
\end{equation}
\end{defi}

In problem \eqref{problem} one must find not only the solution $u$ but also
the speed $c$, which is apriori unknown. We will establish that there is a unique
speed $c\in\R$ for which  \eqref{problem} admits a solution $u$ satisfying the
limits \eqref{limits0}. For such speed $c$,
using variational techniques we show the existence of a
solution $u$ which is decreasing in $y$, with limits $1$ and $0$ at
infinity on every vertical line. Moreover, we prove the 
uniqueness (up to translations in the $y$ variable) of a solution
$u$ with limits 1 and 0 as $y\to\mp\infty$.

The speed $c$ of the front will be shown to
be positive. Hence, since $c>0$, we have that 
$v(x,y,t)=u(x,y-ct) \rightarrow 1$ as 
$t \rightarrow +\infty$. That is, the state $u\equiv 1$ invades the state $u\equiv0$.

For non-balanced 
bistable nonlinearities which satisfy in addition $f'(0)<0$ and $f'(1)<0$,
we find the behaviors of the front at $y=\pm \infty$.
In contrast with the classical bistable interior 
reaction model, its behavior at the side of the invading state $u=1$, i.e. as $y\to -\infty$,
is of power type, while its decay is exponential as $y\to +\infty$.

Our results are collected in the following result. We point out that,
since $f\in C^{1,\gamma}$, weak solutions to problem \eqref{problem} 
can be shown to be classical, indeed $C^{2,\gamma}$ up to $\partial\RR$.
This is explained in the beginning of section~\ref{Limites}.

\begin{theo}
\label{Main} Let $f$ be of positively-balanced bistable type or of combustion type 
as in Definition {\rm \ref{nonlinearterm}}.
We have:

{\rm{(i)}}
There exists a solution pair $(c,u)$ to problem \eqref{problem}, where $c>0$, $0<u<1$,
and $u$ has the limits \eqref{limits0}. The solution $u$ lies in the weighted Sobolev space 
$$
H^1_c (\RR):= \{w \in H_{\rm loc}^{1}(\RR) \, : \, ||w||_{c}:=
 \int_{\RR} e^{cy}\{w^2+|\nabla w |^2\} dxdy <\infty \}.
$$

{\rm{(ii)}} Up to translations in the $y$ variable, 
$(c,u)$ is the unique solution pair to problem \eqref{problem}
among all constants $c\in\R$ and solutions $u$ satisfying 
$0\leq u\leq 1$ and the limits \eqref{limits0}.

{\rm{(iii)}} For all $x\geq 0$, $u$  is decreasing in the $y$ variable, 
and has limits $ u(x,-\infty)= 1$ and $u(x, + \infty) = 0$.
Besides, $\lim_{x\to +\infty} u(x, y)=0$ for all $y\in \R$.
If $f$ is of combustion type then we have, in addition, $u_x\leq 0$ in $\RR$.

{\rm{(iv)}} If $f_1$ is of positively-balanced bistable type or of combustion type, if the same holds
for another nonlinearity $f_2$, and if we have that $f_1\geq f_2$ and $f_1\not
\equiv f_2$, then their corresponding speeds satisfy $c_1>c_2$.

{\rm{(v)}} Assume that $f$ is of positively-balanced bistable type and that
\begin{equation}\label{negder}
f'(0)<0 \quad	\text{and}\quad f'(1)<0.
\end{equation}
Then, there exists a constant $b>1$ such that:
\begin{align}
\label{asyderpos}
\frac{1}{b}\,\,    \frac{e^{-cy}}{y^{3/2}}  \leq -u_y(0,y)  \leq b\,\,  \frac{e^{-cy}}{y^{3/2}} 
& \qquad\text{ for } y>1,  \\
\label{asyderneg}
\frac{1}{b}\, \,   \frac{1}{(-y)^{3/2}}  \leq -u_y(0,y)  \leq b\,\,    \frac{1}{(-y)^{3/2}} 
& \qquad\text{ for }y<-1, \\
\label{asypos}
\frac{1}{b}\,\,   \frac{e^{-cy}}{y^{3/2}}  \leq u(0,y)  \leq b\,\,    \frac{e^{-cy}}{y^{3/2}} 
& \qquad\text{ for } y>1, \quad\text{and}  \\
\label{asyneg}
\frac{1}{b}\, \,   \frac{1}{(-y)^{1/2}}  \leq 1-u(0,y)  \leq b\, \,   \frac{1}{(-y)^{1/2}} 
& \qquad\text{ for }y<-1.
\end{align}

The lower bounds for $-u_y$, $u$, and $1-u$
in \eqref{asyderpos}-\eqref{asyneg} hold
for any $f$ of positively-balanced bistable type or of combustion type as in Definition {\rm \ref{nonlinearterm}}.
\end{theo}

Our result on the existence of the traveling front will be proved using a variational method
introduced by Steffen Heinze~\cite{H}. It is explained later in this section.
Heinze studied problem \eqref{problem} in infinite cylinders of $\R^n$ instead of
half-spaces.  For these domains and for both bistable and combustion nonlinearities, 
he showed the existence of a 
traveling front. Using a rearrangement technique after making the change
of variables $z=e^{ay}/a$, he also proved the monotonicity of the front. 
In addition, \cite{H} found an interesting formula,  \eqref{formspeedintro},
for the front speed in terms of the minimum value of the variational problem.
The formula has interesting consequences, such as part (iv) of
our theorem: the relation between the speeds for two comparable nonlinearities.

For our existence result we will proceed as in \cite{H}. The weakly lower semicontinuity
of the problem will be more delicate in our case due to the unbounded character
of the problem in the $x$ variable ---a feature not present in cylinders.
The rearrangement technique will produce a monotone front. Its monotonicity will
be crucial in order to establish that it has limits 1 and 0 as $y\to\mp\infty$.
On the other hand, the front being in the weighted Sobolev space $H^1_c(\RR)$
will lead easily to the fact that $u\to 0$ as $x\to +\infty$.
While $u_x\leq 0$ in case of combustion nonlinearities ---as stated
in part (iii) of the theorem---, this property
is not true for bistable nonlinearities since the normal derivative $-u_x=f(u)$ changes
sign on $\{x=0\}$.

The variational approach has another interesting feature. Obviously, 
the solutions that we produce
in the half-plane are also traveling fronts for the same problem in a half-space $\R^n_+$, with
$n\geq 3$. They only depend on two Euclidean variables. However, if the minimization problem
is carried out directly in  $\R^n_+$ for $n\geq 3$, 
then it produces a different type of solutions that decay to 0
in all variables but one; see Remark~\ref{highd} for more details. 

In the case of combustion nonlinearities, problem \eqref{problem} in a half-plane 
has been studied by Caffarelli, Mellet, and Sire~\cite{CMS}, 
a paper performed at the same time as most of our work.  
They establish the existence of a speed admitting a monotone front.  As mentioned in 
\cite{CMS}, our approaches towards the existence result are different. 
Their work does not use 
minimization methods, but instead approximation by truncated
problems in bounded domains  ---as in \cite{BN2}. They also rely in an interesting 
explicit formula for traveling fronts of a free boundary problem obtained as a
singular limit of \eqref{problem}. In addition, \cite{CMS}
establishes the following precise behavior of the
combustion front at the side of the invaded state $u=0$.
For some constant $\mu_0>0$,
\begin{equation}\label{theirs}
u(0,y) = \mu_0\, \frac{e^{-cy}  }{y^{1/2}}  +{\rm O}
\left(\frac{e^{-cy}  }{y^{3/2}} \right) \quad \mbox{ as $y\rightarrow +\infty$}
\end{equation}
(here we follow our notation; \cite{CMS} reverses the states
$u=0$ and $u=1$). Note that the decay for combustion
fronts is different than ours: $y^{1/2}$ in \eqref{theirs} is replaced
by $y^{3/2}$ in the bistable case. Note however that the main order in the decays
is $e^{-cy}$ and 
that the exponent $c$ depends in a highly nontrivial way on each nonlinearity~$f$.

Uniqueness issues for the speed or for the front in problem \eqref{problem} 
are treated for first time in the present paper. 
Our result on uniqueness of the speed and of the front 
relies heavily on the powerful sliding method of Berestycki and Nirenberg \cite{BN1}. 
We also use a comparison principle analogue to one in the paper by 
the first author and Sol\`a-Morales \cite{CSo}, 
which studied problem \eqref{problem} with $c=0$.
Among other things, \cite{CSo} established the existence, uniqueness,
and monotonicity of a front for \eqref{problem} when $c=0$
and $f$ is a balanced bistable nonlinearity. It was shown also there that in the
balanced bistable case, the front reaches its limits 1 and 0 at the 
power rate $1/|y|$. We point out that the variational method in the present 
paper requires $f$ to be non-balanced. It cannot be carried out in the balanced case.

Suppose now that $f$ satisfies the assumptions made above 
for bistable nonlinearities except for condition \eqref{posint}, and assume instead that 
$$
\int_0^1f(s)ds \leq 0. 
$$
First, if the above integral is zero (i.e., $f$ is balanced), \cite{CSo} established the existence
of a monotone front for $f$ with speed $c=0$.
Suppose now that the above integral is negative. Then, the
nonlinearity  $\tilde f(s) := -f(1-s)$ has now positive integral and is of bistable type. Thus, it  
produces a solution pair $(\tilde c,\tilde u)$ for problem \eqref{problem}
with positive speed $\tilde c$. Then, if $u(x,y):= 1-\tilde u(x,-y)$,  $(-\tilde c,u)$ 
is a solution pair to \eqref{problem} for  the original $f$. 
The traveling speed $-\tilde c$ is now negative.

To prove our decay estimates as $y\to\pm\infty$, we use ideas from \cite{CMS} and \cite{CSi}.
The estimates rely on the construction of a family
of explicit fronts for some bistable nonlinearities. Their formula and properties are stated in 
Theorem~\ref{theoexpl} below. To explain how we construct these fronts, note
that $u$ is a solution of \eqref{problem} if and only if its trace $v(y):=u(0,y)$ solves
the fractional diffusion equation
\begin{equation}\label{fracpb}
(-\partial_{yy} -c\partial_y)^{1/2} v = f(v)\,\, \text{ in }\R, \quad\text{ for } v(y):=u(0,y).
\end{equation}
This follows from two facts. First, if $u$ solves the first equation in \eqref{problem},
then so does $-u_x$. Second, we have
$(-\partial_x)^2=\partial_{xx}=-\partial_{yy} -c\partial_y$. Note that our main result
states that there is a unique $c\in\R$ for which the fractional equation
\eqref{fracpb} admits a solution connecting 1 and 0.

As in the paper by the first author and Sire \cite{CSi}, which studied problem
$(-\partial_{yy})^{s} v = f(v)$ in $\R$ for balanced bistable nonlinearities,
the construction of explicit fronts will be based on the fundamental solution 
for the homogeneous heat equation associated to the fractional operator 
in \eqref{fracpb}, that is, equation 
$$
\partial_t v + (-\partial_{yy} -c\partial_y)^{1/2} v=0.
$$
The process to find such heat kernel uses an idea from the paper
\cite{CMS} by Caffarelli, Mellet, and Sire, and it  is explained in section \ref{explicit} below.

Regarding the resulting decays \eqref{asypos} and \eqref{asyneg}
for bistable fronts, note that the exponential decay 
at the side of the invaded state $u=0$  is much faster than the power decay \eqref{asyneg}
at the side of the invading state $u=1$. 
This large difference of rates is clearly seen in the explicit fronts that we built.
See Figure~\ref{dibujoexpl} for the plots of one of such fronts,
where the much steeper decay on the right is clearly appreciated.

These decays are also in contrast with the classical ones for the bistable equation
$u_{yy}+cu_y+f(u)=0$ in $\R$, which are both pure exponentials ---with exponents that may
be different at $+\infty$ and $-\infty$. Note however that the exponent $c$ 
in the exponential  term at $+\infty$ for our problem will be different, in general, 
than the corresponding exponent in the classical case ---taking here
the same nonlinearity~$f$ for both problems.

The next theorem concerns the explicit bistable fronts that we construct.
They will lead to the decay bounds of Theorem~\ref{Main} for general bistable fronts. 
They involve the modified Bessel function of the second kind $K_1$ with index $\nu=1$.
We recall that $K_1(s)$ is a positive and decreasing function of $s>0$
(see \cite{AS}).

\begin{theo}\label{theoexpl}
For every $c>0$ and $t>0$, let
$$
u^{t,c}(x,y):=u^t\left( \frac{c}{2}x, \frac{c}{2}y \right),
$$
where
\begin{equation*}
u^t(x,y) := \int_y^{+\infty} e^{-z}\dfrac{x+t}{\pi\sqrt{(x+t)^2+z^2}} 
\,\, K_1( \sqrt{(x+t)^2+z^2} ) dz
\end{equation*}
and $K_1$ is the modified Bessel function of the second kind with index $\nu=1$.

Then, there exists a nonlinearity $f^{t,c}$ of positively-balanced bistable type for which
$(c,u^{t,c})$ is the unique solution pair to problem \eqref{problem}  
with $0\leq u^{t,c}\leq 1$ satisfying the limits \eqref{limits0}. In addition, we have
$$
(f^{t,c})'(0)= (f^{t,c})'(1)= -\frac{c}{2t}. 
$$
and that $f^{t,c}=(c/2)f^t$ for a nonlinearity $f^t$ independent of $c$.

Furthermore, on $\partial\RR$ the derivative of $u^{t,c}$ satisfies
\begin{align*}
& -u^{t,c}_y(0,y) =\frac{t}{(\pi c)^{1/2}}\,\, \frac{e^{-cy}}{y^{3/2}}+ {\rm o}
\left( \frac{e^{-cy}}{y^{3/2}} \right) 
\quad \mbox{ as $y\to +\infty$}, \quad\text{and}\\
& -u^{t,c}_y(0,y) =\frac{t}{(\pi c)^{1/2}}\,\, \frac{1}{(-y)^{3/2}}+ {\rm o}
\left( \frac{1}{(-y)^{3/2}} \right) \quad \mbox{ as $y\to -\infty$.}
\end{align*}
\end{theo}

Figure~\ref{dibujoexpl} shows plots of the explicit bistable front
$u^1=u^{1,2}$, a front with speed $c=2$. In both of them we have $-22\leq y\leq 2$.
The much faster decay for positive 
values of $y$ than for negative values is clearly appreciated.
In (B), the steepest profile corresponds to $x=0$ while the profile 
in the back of the picture is for $x=2$.

\begin{figure}
\centering 
        \begin{subfigure}[b]{0.38\textwidth}
                \includegraphics[width=\textwidth]{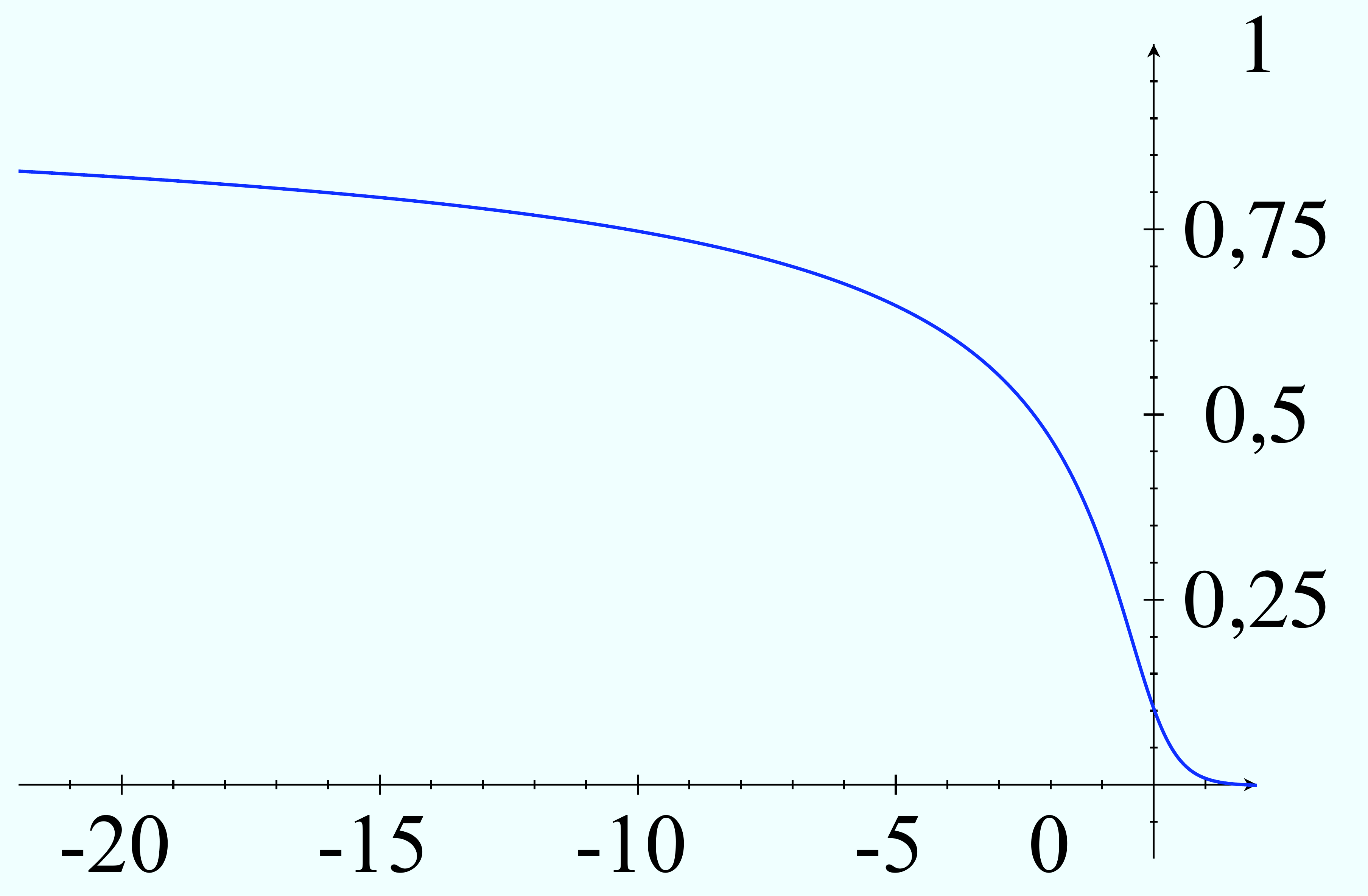}
                \caption{$-22\leq y\leq 2$ and $x=0$}
        \end{subfigure}
        \hspace{-3mm}
           \begin{subfigure}[b]{0.62\textwidth}
                \includegraphics[width=\textwidth]{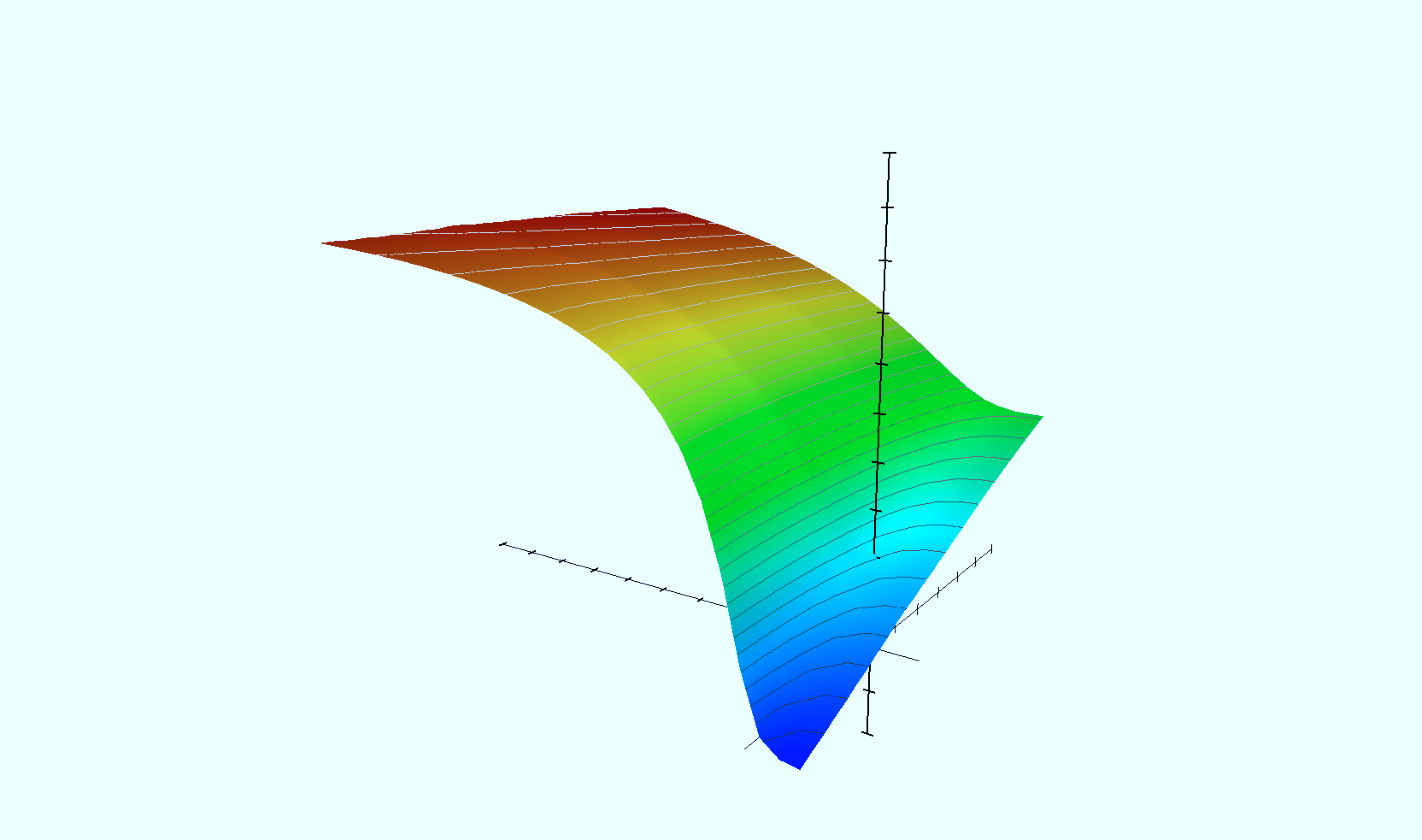}
                \caption{$-22\leq y\leq 2$ and $0\leq x\leq 2$}
        \end{subfigure}
\caption{The explicit bistable front $u^1$}
\label{dibujoexpl}
\end{figure}

The paper by Kyed~\cite{Ky} also studies problem~\eqref{problem} in infinity cylinders of $\R^3$. 
It deals with nonlinearities $f$ that vanish only at 0 and that appear in models of boiling processes.  \cite{Ky} also uses the variational principle of Heinze. In addition,
\cite{Ky} contains some exponential decay bounds.
The papers \cite{L2,L3} by Landes study problem~\eqref{parabolic} in finite cylinders,
with special interest on bistable nonlinearities. These articles establish  
the presence of wave-front type solutions for some initial conditions and  
give in addition bounds for their propagation speed. For this, appropriate sub and 
supersolutions are constructed.

The variational method in \cite{H} has also been used by 
Lucia, Muratov, and Novaga \cite{LMN} to study
the classical interior reaction equation $u_{yy}+cu_y+f(u)=0$ for monostable type
nonlinearities $f$. Their paper gives a very interesting characterization
of the phenomenon of linear versus nonlinear selection for the front speed.

In relation with fractional diffusions ---such as \eqref{fracpb}---,
the existence of traveling fronts for
\begin{equation}\label{fracMRS}
\partial_t v + (-\partial_{yy})^{s} v = f(v) \quad\text{ in }\R
\end{equation}
has been established  in \cite{MRS} when $s\in(1/2,1)$ and $f$ is a combustion nonlinearity. 
This article also shows that $v$ tends to 0 at $+\infty$ at the power rate $1/|y|^{2s-1}$.
Note that the equation for traveling fronts of \eqref{fracMRS} is
$$
\left\{ (-\partial_{yy})^{s} -c\partial_y\right\} v = f(v) \quad\text{ in }\R,
$$ 
which should be compared with \eqref{fracpb}. 
In the case of bistable nonlinearities, \cite{GZ} establishes that 
\eqref{fracMRS}  admits a unique traveling front and a unique
speed for any $s\in (0,1)$. In contrast with the decay in \cite{MRS}
for combustion nonlinearities, in the bistable case
\cite{GZ} shows that the front reaches its two limiting values at the rate
$1/|y|^{2s}$ ---as in \cite{CSo, CSi} for balanced bistable nonlinearities.

Next, let us describe the structure of the variational problem that will lead to our
existence result. First, we enumerate the five concrete properties of the 
nonlinearity needed for all the results of the paper to be true.

\begin{rema}\label{nonl}
{\rm
Even that we state our main result for bistable and combustion type nonlinearities (for the
clarity of the reading), all our proofs require only the following five conditions on~$f$:
\begin{equation}\label{fnegint}
\text{\eqref{zeroes}, \eqref{noninc}, \eqref{posint}, \eqref{fposbeta}, and }
\int_0^s f(\sigma) d\sigma \leq 0 \, \text{ for all } s\in (0,\beta).
\end{equation}
We claim that both positively-balanced bistable nonlinearities 
and combustion nonlinearities as in Definition~\ref{nonlinearterm} satisfy the above five assumptions.

The claim is easily seen. For a combustion nonlinearity, \eqref{posint} is obviously true,
while the last condition in
\eqref{fnegint} holds (indeed with an equality) for the same $\beta$ as in part (b)
of the definition. On the other hand, for $f$ of bistable type, since $f$ has a unique zero
$\alpha$ in $(0,1)$ and \eqref{noninc} holds, it follows that $f<0$ in $(0,\alpha)$ and 
 $f>0$ in $(\alpha,1)$. Thus, by \eqref{posint} there exists a unique $\beta\in (\alpha,1)$ such that
 $\int_0^\beta f(s)ds=0$. As a consequence, \eqref{fposbeta} and the last condition in
 \eqref{fnegint}  hold for such $\beta$.
}
\end{rema}

To describe the potential energy of our problem,
we first extend $f$ linearly to $(-\infty,0)$ and to $(1,+\infty)$, keeping
its $C^{1,\gamma}$ character.  Consider now the potential $G \in C^2(\R)$ defined by  
$$
G(s):=-\int_0^s f(\sigma)d\sigma \qquad\text{for } s\in\R.
$$
Note that $G'=-f$ in $[0,1]$. 
Since $f'(0)\leq 0$ and $f'(1)\leq 0$ due to hypothesis \eqref{noninc}, we have that
\begin{equation}
\label{Goutside}
G(s)\geq
\begin{cases}
0=G(0)  &\text{ for }  s \leq 0\\ 
G(1) &\text{ for } s \geq 1.
\end{cases}  
\end{equation}

Two other important properties of $G$ are the following. First, by \eqref{posint}, we have
\begin{equation}
\label{1negener}
G(1)<G(0)=0\qquad	\text{and}\qquad G'(0)=-f(0)=0.
\end{equation}
Second,  the last condition in \eqref{fnegint} reads 
\begin{equation}
\label{Gpos}
G\geq 0 \quad\text{ in } [0,\beta].
\end{equation}
On the other hand, since $G \in C^2(\R)$ and $G(0)=G'(0)=0$, we have that 
\begin{equation}
\label{cotasG}
 -Cs^2 \leq G(s) \leq Cs^2  \quad \textrm{ for all } s \in \R ,
\end{equation}
for some constant $C$. 

Figure \ref{potential} shows the shape of the potential $G$ for a typical positively-balanced
bistable nonlinearity.

\begin{figure}[h]
\centering 
\includegraphics{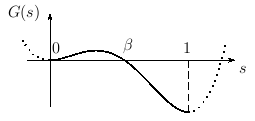}
\caption{The potential $G$ for a positively-balanced bistable $f$} 
\label{potential}
\end{figure}

For $a>0$, consider the weighted Sobolev space $(H^1_a (\RR), ||\, \,||_{a})$ defined by 
$$ H^1_a (\RR)= \{w \in H_{\rm loc}^{1}(\RR) \, : \, ||w||_{a}<\infty \},$$
where  the norm $||\, \, ||_{a}$ is defined by
\begin{equation*}
 ||w||^2_a =
 \int_{\RR} e^{ay}\{w^2+|\nabla w |^2\} dxdy.
\end{equation*}
Notice that, by first truncating and then smoothing, the set
 $ C_c^\infty (\overline{\RR})$ 
---smooth functions with compact support in $\overline{\RR}$--- is dense in  $H^1_a (\RR)$ with the $||\, \, ||_{a}$ norm.

In both the bistable and combustion cases, the 
traveling front $u$ will be constructed  
from a  minimizer $\underline{u}$ to the constraint problem
\begin{equation}
\label{problemconstraint}
E_a(\underline{u})= \INF =:I_a,
\end{equation}
after scaling its independent variables $x$ and $y$. This is the method introduced
by Heinze~\cite{H} to study problem \eqref{problem} in cylinders instead of half-spaces.
The energy functional  
\begin{equation}
\label{FormulaEnergia}
 E_a(w) = \frac{1}{2}\int_{\RR} e^{ay}|\nabla w |^2 dxdy
   + \int_{\FR} e^{ay}G(w(0,y))dy
\end{equation} 
will be minimized over the submanifold 
\begin{equation*}
\BA = \{w \in H^1_a (\RR) : \,  \Gamma_a(w)= 1 \},
\end{equation*} 
where
$$ \Gamma_a(w) = \CIN. $$ 
To carry out this program, we will need to take the constant $a>0$ small
enough depending only on $f$. 

Note an important feature of the functionals $E_a$ and $\Gamma_a$. 
For $w \in \Ha$ and  $t \in \R$,  define $$w^t(x,y):=w(x,y+t)$$
(throughout the paper there is no risk of confusion with the same notation 
used for the explicit front $u^t$ of Theorem~\ref{theoexpl}). We then have
\begin{equation}
 \label{escalaEnergia}
E_a(w^t)=e^{-at}E_a(w)\quad \textrm{and}\quad \Gamma_a(w^t)=e^{-at}\Gamma_a(w).
\end{equation}

The shape of the potential $G$
will lead to the existence of functions $u$ in $H^1_a (\RR)$ 
with negative energy  $E_a(u)<0$. This will be essential in order to prove that
our variational problem attains its infimum. In addition, the constraint will introduce
a Lagrange multiplier and, through it, the apriori unknown speed $c$ of the traveling front. 

As already noted by Heinze~\cite{H}, the above 
variational method produces an interesting formula for the speed $c$.
One has
\begin{equation}\label{formspeedintro}
c=a(1-2I_a),
\end{equation}
where $I_a$ is the minimum value \eqref{problemconstraint}
of the constraint variational problem;  see Remark~\ref{mirema} below.  
This formula leads to the comparison result
between the front speeds for two different comparable nonlinearities ---part (iv) of 
Theorem~\ref{Main}. Note also that the value $a(1-2I_a)$ in \eqref{formspeedintro}
does not depend on which constant $a$ is chosen to carry out the minimization problem.
The reason is that this value coincides with the speed~$c$, and we prove uniqueness of
the speed.
  
\begin{rema}\label{highd}
{\rm
Obviously, the traveling front found in our paper on $\RR$ is also a traveling front
for problem \eqref{parabolic} in the half-space $\RN$, $n\geq 3$, traveling in any given 
unit direction $e$
of $ \R^{n-1}$. That is, it is a solution of the problem
$$\begin{cases}
\Delta u + c\partial_eu=0 &\text{ in } \mathbb{\R}^n_+ \\
\displaystyle{\frac{\partial u}{\partial\nu}} =f(u) &\text{ on }
\partial\mathbb{\R}^n_+,
\end{cases}$$
where $\R^n_+:= \{(x,y) \in (0,+\infty) \times \R^{n-1}\}$. 
An interesting application of the constraint minimization method is that, 
 when carried out in $\R^n_+$ and $n\geq 3$, it produces another type of traveling fronts. Their trace $u(0,\cdotp)$ 
will not depend only on one Euclidean variable  $e$ of $ \R^{n-1}$ 
(as the solutions in the present paper do), but they will be monotone with limits 
1 and 0 at infinity in the direction $e$ and will be even with limits 
0 at  $\pm\infty$ on the $y$ variables orthogonal to $e$. The reason is that 
here one minimizes the energy functional
$$ E_a(w) = \frac{1}{2}\int_{\R^n_+} e^{a\cdotp y}|\nabla w |^2 dxdy
   + \int_{\partial\mathbb{R}^n_+} e^{a\cdotp y}G(w)dy$$
under the constraint on the Dirichlet energy, where $a \in \R^{n-1}$ is a non-unitary direction parallel to $e$. Now, note that the solutions built in this paper, which are constant in 
the $y$ variables orthogonal to $a$, do not belong to the corresponding 
energy space (since they have infinite Dirichlet energy).
}
\end{rema}

The article is organized as follows. In section \ref{variational} we study the variational
structure of problem
\eqref{problem} and prove the existence of a solution pair.
Section \ref{reordenamiento} uses the variational characterization and a monotone decreasing 
 rearrangement of the minimizer to show that the front may be taken to be monotone
 in the $y$ direction. In section \ref{Limites} we establish the limits  at infinity for the obtained 
 solution.
  In section \ref{unicidad} we prove a monotonicity and comparison result
  by means of a maximum principle and the 
sliding method. This result is the key ingredient to prove uniqueness of speed and of the
front. Section~\ref{explicit} deals with the explicit fronts and supersolutions; here we give the
proof of Theorem~\ref{theoexpl}.  Finally,  
in section \ref{pruebafinal}  we collect all results in the paper to establish
Theorem~\ref{Main}.

\section{The variational solution}
\label{variational}

\subsection{Two inequalities}

We prove  trace and  Poincar\'e-type inequalities for functions in
$H^1_a (\RR)$. The existence of a lower bound for $E_a$ on $B_a$ will be a
consequence of the following lemma.
\begin{lem}
\label{ineq2}
Let $a>0$. Then, for every $ u \in H^1_a (\RR)$, we have:
$$
\int_\R e^{ay} u^2(0,y)dy \leq ||u||_a^2  = \int_{\RR}e^{ay}\{u^2+|\nabla u|^2\} dxdy
$$
and
$$ \int_{\RR}e^{ay}u^2dxdy \leq \frac{4}{a^2}
  \int_{\RR}e^{ay}|\nabla u|^2dxdy. 
$$
\end{lem}

\begin{proof}
By density, it suffices to establish both inequalities  for $ u \in
C^{\infty}_c(\overline{\RR}) $.  Since $u$ has compact support, we have
\begin{equation*}
\begin{split}
\int_{\R}e^{ay}u^2(0,y)dy = &
 - \int_{\RR}(e^{ay}u^2)_x dxdy\\ = & -2\int_{\RR}
e^{ay}uu_xdxdy \leq \int_{\RR}e^{ay}(u^2+u_x^2) dxdy. 
\end{split}
\end{equation*}
 This proves the first inequality.
 
Next, for every $ x \geq 0 $ we have that
$$ \int_{\R}e^{ay}u^2(x,y)dy =
-\dfrac{2}{a}\int_{\R}e^{ay}u(x,y)u_y(x,y) dy. $$\\
Thus, by Cauchy-Schwarz,
$$ \int_{\R}e^{ay}u^2(x,y)dy \leq \dfrac{4}{a^2}
\int_{\R}e^{ay}u_y^2(x,y)dy. $$ 
Integrating in $x$ from  0 to $ \infty$, we obtain the second inequality.
\end{proof}

\subsection{Construction of functions with negative energy}
It is fundamental to show that $E_a$ takes a negative value somewhere on the 
constraint $B_a$. This will be accomplished by
constructing test functions $u_{0}\in H^1_a (\RR)$
for which $E_a(u_{0})<0$ if $a$ is positive and small enough. 
We undertake this task next.

Let $u_{0}$ be defined by 
\begin{equation}
\label{funciondeprueba}
u_{0}(x,y):=e^{-dx}h(y),
\end{equation}
  where  $h$ is given by
$$h(y):=
\begin{cases}
1 & \text{if } y\leq 0\\
e^{-amy} & \text{if } y>0,
\end{cases}$$ and the values of $d>0$ and $m\geq 1$ are to be determined. 

The following proposition applies to a class of nonlinearities which includes those 
 of positively-balanced bistable type and of
 combustion type.
\begin{propo}\label{tecn0}
Let $G$ satisfy \eqref{1negener}.
Let $a>0$, $E_a$ be defined by \eqref{FormulaEnergia}, and $u_{0}$ by \eqref{funciondeprueba}.

Then,  for small positive values of $a$ and $d$, and
 for large values of $m$ $($all depending only on $f)$, we have that 
 $u_0\in  H^1_a (\RR)$ and $E_a(u_{0})<0$.
In addition, 
\begin{equation*}
-\infty<\INF<0.
\end{equation*}
\end{propo}

\begin{proof}
A simple calculation for the Dirichlet energy shows that, for $m\geq 1$,
$$
\Gamma_a(u_{0})=\dfrac{d}{2a}\left(1+\dfrac{1}{2m-1}\right)
+\dfrac{am^2}{2d(2m-1)}.$$ The potential energy can be computed as
follows:
\begin{equation*}
\begin{split}
\int_{\R}e^{ay}G(u_{0}(0,y))dy & =
\dfrac{G(1)}{a}+\int_0^{+\infty}e^{ay}G(e^{-amy})dy\\
&=\frac{1}{a}\left\{ G(1)+\int_0^{+\infty} (e^{ay})'G(e^{-amy}) dy\right\}\\
& = \frac{1}{a}\int_0^{+\infty} e^{ay}f(e^{-amy})(-am) e^{-amy}dy
=-\dfrac{1}{a} \int_0^1 s^{-1/m} f(s) ds,
\end{split}
\end{equation*}
where we have used property \eqref{cotasG} of $G$ in order 
to integrate by parts. Note that \eqref{cotasG} follows from assumption \eqref{1negener}. 
Therefore,
$$
aE_a(u_{0})=\dfrac{d}{4}\left(1+\dfrac{1}{2m-1}\right)
+\dfrac{a^2m^2}{4d(2m-1)}
-\int_0^1 s^{-1/m} f(s) ds. 
$$

Note that $$\lim_{m\rightarrow
+\infty }\int_0^1  s^{-1/m} f(s)  ds = \int_0^1f(s)ds >0,
$$
since $G(1)<G(0)$. It follows that $E_a(u_{0})<0$ if we first choose $m$ large enough,
then choose $d$ small enough to make the first term above small, and finally $a$ 
also small to handle the second term.

It follows from property \eqref{escalaEnergia} that there exists a unique value $t$ such that
$\Gamma_a(u_0^t)=e^{-at}\Gamma_a(u_0) = 1$.  Since  $u_0^t \in B_a$ and
$E_a(u_0^t)=e^{-at}E_a(u_0)<0$, we have shown that  
$$
\INF<0.
$$
As a consequence of the two inequalities in Proposition \ref{ineq2} and of \eqref{cotasG}, 
we obtain  that $E_a$ is bounded by below on $B_a$. Therefore
$$-\infty<\INF<0,
$$
as claimed.
\end{proof}

\subsection{A special minimizing sequence}

To establish that our constraint variational problem \eqref{problemconstraint} achieves its infimum it will be important to work with the following type of minimizing sequences.

\begin{lem}
\label{acotacionvariable}
Let $G$ satisfy \eqref{Goutside}, \eqref{1negener}, and \eqref{cotasG}. Then, 
there exists a minimizing sequence $\{u_{k}\}\subset B_{a}$ of problem
\eqref{problemconstraint}
such that, for all $k$, $u_{k}\in C_c(\overline{\RR})$ 
has compact support in $\overline{\RR}$ and $0\leq u_{k}\leq1$.
\end{lem}
 
\begin{proof}
By the previous proposition we know that $$-\infty <I_a:= \INF <0.$$ 
Let $\{w_k\} \subset B_a$ be any minimizing sequence. Approximating each $w_k$ 
by a function $v_k$ in $C_c^\infty (\overline{\RR})$ (by first truncating and then smoothing), we may assume that
$\lim_{k\rightarrow +\infty}E_a(v_k)=I_a$ and $\Gamma_a(v_k)=1+ \tau_k$
with $\tau_k \to 0$. Consider now $t_k= \log \{(1+\tau_k)^{1/a}\}$. 
As a consequence of \eqref{escalaEnergia}, we obtain $\Gamma_a(v_k^{t_k})=1$ ---thus 
$\{v_k^{t_k}\} \subset B_a \cap C_c^\infty (\overline{\RR})$---  and $\lim_{k\rightarrow +\infty}E_a(v_k^{t_k})=I_a.$

Next,  to show that we can restrict ourselves to minimizing sequences taking values in $[0,1]$,
let us rename $\{v_k^{t_k}\}$ by $\{v_k\}$. We truncate $\{v_{k}\}$ and
define $\tilde{v}_{k}$ by
$$
 \tilde{v}_k=
\begin{cases}
0 & \text{ if } v_k<0 \\
v_k & \text{ if } v_k\in[0,1]\\
1 & \text{ if } v_k>1.
\end{cases}
$$
It is easy to see that $\tilde{v} _{k}\in C_{c}(\overline{\RR}) \cap H^1_a(\RR)$,
  $\Gamma_a(\tilde{v}_k)\leq \Gamma_a(v_k)$ and, using \eqref{Goutside}, 
  $E_a(\tilde{v}_k)\leq E_a(v_k).$ 
  
Next we claim that we may choose $s_k\leq0$ so that
$u_k(x,y):=\tilde{v}_k(x,y+s_k)$ satisfies $\{u_k\} \subset B_a$. 
This claim follows from \eqref{escalaEnergia} and the fact that 
$0<\Gamma_{a}(\tilde{v}_{k}) \leq 1$ for $k$ large.  
To show this last assertion, note that
$\Gamma_{a}(\tilde{v}_{k}) \leq 1$ is a consequence of $\Gamma_{a}(\tilde{v}_{k}) 
\leq \Gamma_{a}(v_{k}) = 1$. 
On the other hand, if $\Gamma_{a}(\tilde{v}_{k}) =0$ then 
 $\tilde{v}_k\equiv 0$ and thus $v_k\leq 0.$ 
 From this and \eqref{Goutside},  we would get $E_a(v_k)\geq 0$.
 This is a contradiction if $k$ is large, since
 $I_a<0$ and $\{v_k\}$ is a minimizing sequence. 

Finally, notice that since $s_{k}\leq 0$ and $E_{a}(\tilde{v}_{k})
\leq E_a(v_k)<0$ for $k$ large, \eqref{escalaEnergia} gives 
$E_a(u_k) \leq E_a(\tilde{v}_k) \leq E_a(v_k)$. Therefore, $\{u_k\}\subset B_{a}$ is a 
 minimizing sequence made of continuous functions with compact support
and satisfying $0\leq u_k \leq 1$.
\end{proof}

\subsection{Weak lower semi-continuity}
Due to the unbounded character of $\RR$, a delicate issue in this paper is to prove the  weak lower semi-continuity (WLSC) of $E_a$ in~$B_a$. 
The key point is to 
establish the WLSC result for the potential energy ---the difficulty being that the
potential $G(s)$ is negative near $s=1$.

Note that  for any 
sequence $\{u_k\} \in H^1_a (\RR)$ converging weakly in $H^1_a(\RR)$ to a
 function $\underline{u}$,  $u_k\rightharpoonup \underline{u}$, it holds
$\Gamma_a(\underline{u}) \leq \liminf_{k\rightarrow \infty}
\Gamma_a(u_k)$. 
Also, if we split $G= G^+ - G^-$ into its positive
and negative parts, Fatou's lemma gives 
$$ \int_{\partial \RR}e^{ay}G^+(\underline{u})dy \leq \liminf_{k \rightarrow
\infty} \int_{\partial \RR}e^{ay}G^+(u_k)dy. $$
Thus, we need to study the convergence of 
$$\int _{\partial \RR} e^{ay}G^-(u_k)dy.$$

To do this, the key observation (that already appears in \cite{H}) is that any
minimizing sequence cannot spend ``too much time'' (time meaning ``positive 
$y$ variable'') in  $(\beta,1)$, where $G$ may be negative, 
even if the state $u=1$ invades $u=0$.  This key fact will be a consequence of
the presence of the weight $e^{ay}$. We will see that, 
for any given $R>0$ and any minimizing sequence $\{u_{k}\}$, the 
Lebesgue measure of the sets $\{ y> R \, : \, u_{k}(0,y)\geq \beta\}$ and
$\{ x>0, y>R\, : \,u_{k}(x,y)\geq \beta\}$ 
both decrease to zero as $R\rightarrow +\infty$, uniformly in $k$.

To proceed from this, our analysis must be more
delicate than in \cite{H} ---which deals with cylinders---, 
due to the unbounded character of $\RR$ in the $x$ variable. To
handle this difficulty, we need
to recall some facts about Riesz potentials; see \cite[section 7.8]{GT}.
Let $ \Omega$ be a bounded domain in $\R^2$.
Consider the operator on $L^2(\Omega)$  defined by
\begin{equation*}
Vw(z):= \int_{\Omega}\dfrac{w(\overline{z})}{|z-\overline{z}|}d\overline{z}.
\end{equation*}
It is well known (see \cite[section 7.8]{GT}) that 
\begin{equation}
\label{CotaRiesz}
 ||Vw||_{L^2(\Omega)} \leq 2\sqrt{\pi}|\Omega|^{1/2}||w||_{L^2(\Omega)}. 
\end{equation}

Next, we use this inequality to prove a proposition that will be important
to control the super level sets of $u_k$ mentioned above.
The proposition will be applied to the functions $v= e^{ay/2}(u_k-\beta)^+$.

\begin{propo}
\label{poincareGT} 
Given any constant $a>0$, let $v\in C_c(\overline{\RR})\cap H^{1}_{a}(\RR) $ 
have compact support  with ${\rm supp}(v) \subset {\overline{\Omega}}$ for some
bounded domain  $\Omega\subset\RR$. 
For any $R>0$, define 
$$
\Omega_R:=\Omega \cap \{(x,y)\in \RR \,: \, y>R\}  \quad\text{ and }\quad 
\partial ^0 \Omega_R:=\overline{\Omega_R} \cap \partial \RR.
$$

Then,  we have
$$ ||v||_{L^2(\Omega_R)}
\leq \dfrac{4}{\sqrt{\pi}} |\Omega_R|^{1/2}||\nabla v ||_{L^2(\Omega_R)} $$ 
and 
$$ \left( \int_{\partial^0 \Omega_R}v^2(0,y)dy\right)^{1/2} \leq 
\dfrac{\sqrt 8}{\sqrt[4]{\pi}}|\Omega_R|^{1/4}||\nabla v||_{L^2(\Omega_R)}.
$$
\end{propo}

\begin{proof} 
By density, it is enough to consider $v \in C^{\infty}_{c}(\overline{\RR})$. 
 Since $v$ has compact support,  for $z \in
\Omega_R$ and $\omega=(\omega_1,\omega_2) \in S^1$ with $\omega_1>0$ and $\omega_2>0$, we have 
$$ v(z)= - \int_0^{\infty}D_rv(z+r \omega)dr.$$
Integrating with respect to $\omega$ on the quarter of circle  $$S^1_+:=\{\omega\in S^1
\,:\,\omega_1>0 \textrm{ and }\omega_2>0\},$$ we obtain
$$v(z)= - \frac{2}{\pi}\int_0^{\infty}dr \int_{S^1_+}
D_rv(z+r\omega)d\omega  .$$ This leads to 
$$|v(z)| \leq \frac{2}{\pi}
\int_{\Omega_R}\frac{|\nabla v(\overline{z})|}{|z-\overline{z}|}d\overline{z}.$$ From this, the first inequality
in Proposition \ref{poincareGT} is now a consequence of \eqref{CotaRiesz}.

As for the second 
inequality, we have 
\begin{eqnarray*}
\int_{\partial^0 \Omega_R}v^2(0,y)dy &=& - \int_{ \Omega_R}(v^2)_xdxdy =
 -2 \int_{ \Omega_R}vv_xdxdy \\
& \leq & 2||v||_{L^2(\Omega_R)}||\nabla v||_{L^2(\Omega_R)} \leq
 \dfrac{8}{\sqrt{\pi}}|\Omega_R|^{1/2}||\nabla v||^2_{L^2(\Omega_R)},
\end{eqnarray*}
as claimed.
\end{proof}

Next, note that for any fixed $R >0$, the embedding 
$$H^1_a (\RR\cap \{y < R\})\cap L^\infty(\RR\cap \{y < R\})
 \hookrightarrow L^2_a(\partial \RR\cap
\{y < R\})$$ is compact ($L^2_a$ is the $L^2$ space for the measure $e^{ay}dy$). 
Indeed, a bounded sequence in $H^1_a (\RR\cap \{y < R\})$ 
is also bounded in $H^{1}((0,1)\times (-M,R))$ for 
all $M>0$. This last space 
is compactly embedded in $L^{2}(\FR \cap \{-M <y < R\})$, and thus also in 
$L^{2}_{a}(\FR \cap \{-M <y < R\})$. In addition, since the sequence of functions is 
bounded in $L^{\infty}(\FR)$, their 
$L^{2}_{a}(\FR \cap \{-\infty <y < -M\})$ norms are as small as wished, as
$M\to\infty$ ---since $e^{ay}\leq e^{-aM}$ in this set.

Thanks to the previous compact embedding, to achieve the desired WLSC result for $E_a$, it is
enough to prove that  $$ \int _R^{+\infty} e^{ay}G^-(u_k(0,y))dy$$
can be made ---uniformly on $k$--- as small as we want, provided that $R$
is large enough. This is the content of the next proposition.

\begin{propo}
\label{WLSC} Let $G$ satisfy \eqref{Gpos}, i.e., $G\ge 0$ in 
$[0,\beta]$.  Let $\{u_k\}\subset B_a \subset H^1_a (\RR)$ be a minimizing
sequence for $\inf \{E_a(w)\, : \, w\in B_a\}$  such that 
$u_k\in C_c(\overline{\RR})$ and $0\leq u_k\leq 1$ for all $k$. 
 
 Then, given  $\varepsilon >0$,  there exists $R>0$ such that  we have $$\int
_R^{+\infty} e^{ay}G^-(u_k(0,y))dy \leq \varepsilon $$
for all $k$.
\end{propo}

\begin{proof}
For  $R>0$, we define $$ A_k := \{(x,y)\in \RR \, : \,y>R, \,
u_k(x,y)> \beta \}. $$

We can estimate the 
measures of  $\partial ^0 A_k$ and $A_k$ respectively as follows
(recall that the notation $\partial^0$ was introduced in Proposition~\ref{poincareGT}). 
First,
$$ e^{aR} |\partial ^0 A_k| \leq \int _{\partial ^0 A_k}e^{ay}dy
\leq \frac{1}{\beta ^2}\int _{\partial ^0 A_k}e^{ay}u_k^2(0,y)dy
\leq \frac{C}{\beta ^2}.$$ The last inequality is a consequence of
$\Gamma_a(u_k)=1$ and the trace inequality in Lemma~\ref{ineq2}. 
In what follows, $C$ denotes different constants
 depending only $a$ (and thus, not on $k$).
Therefore we have $$ |\partial ^0 A_k| \leq \frac{C}{\beta ^2}e^{-aR}$$ and
\begin{equation}
\label{trazaenAn}
 \int _{\partial ^0 A_k}e^{ay}dy \leq \frac{C}{\beta
^2}.
\end{equation}
In an analogous way ---integrating now on all of $A_k$ and not on its boundary,
 and using again Lemma~\ref{ineq2}--- we obtain that 
\begin{equation}
\label{medidaAn}
 |A_k| \leq \frac{C}{\beta ^2}e^{-aR}.
\end{equation}

By \eqref{Gpos}, there exists a constant $C$ such that
$G^-(s)\leq C(s-\beta)^+$ for $s \in [0,1]$.
This and $0\leq u_k\leq1$ lead to
\begin{eqnarray*}
\int_R^{+\infty}e^{ay}G^-(u_k(0,y)) dy &=& \int_{\partial ^0 A_k}e^{ay}G^-(u_k)dy
\leq C \int_{\partial ^0 A_k}e^{ay}(u_k-\beta) dy \\
&\leq & C \left(\int_{\partial ^0 A_k}e^{ay}(u_k-\beta)^2 dy\right)^{1/2} \left(\int_{\partial ^0 A_k}
e^{ay}dy\right)^{1/2}.
\end{eqnarray*}
Because of \eqref{trazaenAn} above,
the last factor is bounded by $C/\beta$, 
a constant independent of~$k.$ Using the second inequality
 in Proposition \ref{poincareGT} applied to the function $e^{ay/2}(u_k-\beta)^+$, we get
\begin{equation}
\label{GnoFatou}
 \int_R^{+\infty}e^{ay}G^-(u_k(0,y))dy \leq
\frac{C}{\beta}|A_k|^{1/4}\left(\int_{A_k}|\nabla
(e^{ay/2}(u_k-\beta))|^2dxdy\right)^{1/2}.
\end{equation}

Using Cauchy-Schwartz and the trace inequality of Lemma \ref{ineq2}, we see
that the integral on the right hand side of \eqref{GnoFatou}
is bounded by a constant independent of $k$. In addition,
as a consequence of inequality \eqref{medidaAn} we have that  
$$\lim_{R\rightarrow \infty} |A_k|= 0.$$ 
Thus, the result follows from \eqref{GnoFatou}.
\end{proof}

We can now show that the infimum is achieved.
\begin{coro}
\label{exist2d} Let $f$ be  of positively-balanced bistable type or of combustion type as in Definition {\rm \ref{nonlinearterm}}.
Then, for every  $a>0$ small enough $($depending only on $f)$, 
there exists  $ \underline{u} \in \BA $ such
that
$$E_a(\underline{u})= \INF.$$
In addition  $0\leq \underline{u}\leq 1$,
$|\{\underline{u}(0,y) : y \in \R\} \setminus [0,\beta] |>0$, and
$\underline{u}$ is  not identically constant.
\end{coro}

\begin{proof}
For $a>0$ small enough, Proposition \ref{tecn0} shows that
that $-\infty< I_a < 0$,  where 
\begin{equation}
\label{mimin}
I_a:= \INF.
\end{equation} 
By Lemma \ref{acotacionvariable}, there exists $\{u_k\} \subset B_a \subset H^1_a (\RR)$ such
that $ E_a(u_k) \rightarrow I_a$, $u_k\in C_c(\overline\RR)$, and $0\leq u_k\leq 1$. 
Since $\{u_k\} \subset \BA$, by Lemma~\ref{ineq2}
$\{u_k\}$ is bounded in $H^1_a (\RR)$. Therefore, there exists a weakly convergent
 subsequence (still denoted by $\{u_k\}$) such that $u_k\rightharpoonup \underline{u}$ and 
 $\underline{u} \in  H^1_a (\RR)$.  
 
By the WLSC comments made on the beginning of this
subsection on the kinetic energy and the potential energy corresponding to $G^+$,
and by the compactness result of Proposition~\ref{WLSC}, we have that 
$E_a(\underline{u})\leq \liminf_k E_a(u_k)$. Thus, we will have that 
$\underline{u}$ is a minimizer if we show that
$$\underline{u} \in \BA.$$

To show this claim, recall that $\Gamma_a(\underline{u}) \leq 
\liminf_k \Gamma_a(u_k) = 1$.
If we had $\Gamma_a(\underline{u})=0$ then $\underline{u} \equiv 0$,
Thus, we would have
$0 = E_a(\underline{u})\leq \liminf_k E_a(u_k)=I_a<0$, a contradiction.
Hence, $\Gamma_a(\underline{u})\in(0,1]$. 

Let us see now that $\Gamma_a(\underline{u})
\in (0,1)$ is not possible either. Indeed,  
assume that $\Gamma_a(\underline{u}) < 1$. Then, for some $t < 0$,
the function $\underline{u}^t$ defined by $\underline{u}^t(x,y):=\underline{u}(x,y+t)$ satisfies
 $\Gamma_a(\underline{u}^t)=e^{-at}\Gamma_a(\underline{u})=1$, and hence
 $\underline{u}^t\in
\BA $. In addition, $ E_a(\underline{u}^t)=e^{-at}E_a(\underline{u})< E_a(\underline{u})= I_a$,
which is a contradiction. Therefore, we have shown our claim $\underline{u} \in \BA. $

To prove the last statements of the corollary, since $0\leq u_k\leq 1$ the same holds for 
$\underline{u}$.  Moreover, since $\underline{u} \in \BA$, $\underline{u}$ is  not identically constant. Finally, 
if we had $|\{\underline{u}(0,y)\; : \; y \in \R\} \setminus [0,\beta]|=0$, then
$E_a(\underline{u})\geq 0$ by \eqref{Gpos} and this is a contradiction. 
\end{proof}

\subsection{Solving the PDE}
In this part we show that there exists 
a solution pair $(c,u)$ to problem \eqref{problem}, with $c>0$ and $u$ not identically constant. 
The solution is constructed from a minimizer $\underline{u}$ of our variational 
problem, after scaling its independent variables $(x,y)$ to take care of 
a Lagrange multiplier $\lambda_a$. The speed turns out to be
$c=a(1-2I_a)=a(1-2\lambda_a)$; see \eqref{formspeed}. 

\begin{propo}
\label{PDE} Let $f$ be  of positively-balanced bistable type
or of combustion type as in Definition {\rm \ref{nonlinearterm}}.
Let $\underline{u}$ be a minimizer for problem \eqref{problemconstraint} as given by 
Corollary~{\rm \ref{exist2d}}.
Then, there exists $c>0$ and $\mu>0$ such that, defining $$ u(x,y)=\underline{u}(\mu x, \mu y),$$ 
we have that  $(c,u)$ is a solution pair for problem~\eqref{problem},
 $u$ is not identically constant, $0 \leq u \leq1$, and $u\in H^1_{ c}(\RR).$
\end{propo}

\begin{proof}
Let $\underline{u} \in B_a$ be a minimizer as in Corollary~\ref{exist2d}.
Since
$$
D \Gamma_a(\underline{u})\cdot\underline{u} =\int_{\RR}2e^{ay}|\nabla \underline{u}|^2
dxdy= 2,$$ we deduce that
$D \Gamma_a (\underline{u})\not\equiv 0.$ Therefore, there exists a Lagrange 
multiplier $\lambda_a \in
\R$ such that $DE_a (\underline{u})\cdot\phi= \lambda_a D \Gamma_a (\underline{u})\cdot\phi$ 
for all
$\phi \in H^1_a (\RR)$, that is,
 \begin{equation}
 \label{lagrange}
 (1- 2\lambda_a)\int_{\RR}e^{ay}\nabla \underline{u} \nabla \phi\, dxdy -
\int_{\partial \RR}e^{ay}f(\underline{u}(0,y))\phi(0,y)\, dy =0.
\end{equation}

Let us see that $\lambda_a \neq 1/2$. Indeed, otherwise, from \eqref{lagrange} we deduce 
$f(\underline{u}(0,\cdot)) \equiv 0$ in $\R$. Thus, by assumption \eqref{fposbeta} on $f$, 
we would  have that  either
  $\underline{u}(0,\cdot)\equiv 1$ or that $0\leq \underline{u}(0,\cdot)\leq\beta$.
  The first possibility is not possible since $\underline{u}\equiv  1\not\in H^1_a(\FR).$ 
  On the other hand,
 $0\leq \underline{u}(0,\cdot)\leq\beta$ is ruled out by the last statement of Corollary~\ref{exist2d}. 

Let us consider arbitrary functions $\varphi \in C^\infty _c(\RR)$ vanishing on $\FR$, and 
also functions $\psi \in C^\infty _c(\overline{\RR})$. From \eqref{lagrange}
and $\lambda_a\neq 1/2$, we have
 \begin{equation*}
\int_{\RR}e^{ay}\{\Delta \underline{u} +a\underline{u}_y\} \varphi \, dx dy =0
\end{equation*}
and
\begin{equation*}
\int_{\partial \RR}e^{ay}\{(1-2\lambda_a)\underline{u}_x+f(\underline{u})\}\psi\, dy = 0.
\end{equation*}
As a consequence, and since $\lambda_a\not =1/2$, the pair
$(a,\underline{u})$ is a solution of
\begin{equation*}
\begin{cases}
\Delta \underline{u} + a\underline{u}_y=0&\text{ in } \R^2_+\\
\displaystyle{\frac{\partial \underline{u}}{\partial\nu}}
=\frac{1}{1-2\lambda_a}f(\underline{u}) &\text{ on }
\partial\R^2_+  .
\end{cases}
\end{equation*}

Let us now show that $\lambda_a < 1/2.$
Consider the test function $(\underline{u}-\beta)^+ \in H^1_a (\RR)$.
Plugging it into \eqref{lagrange} we get
\begin{equation}
\label{signolambda} (1- 2\lambda_a)\int_{\{\underline{u}>\beta \}}e^{ay}|\nabla
\underline{u}|^2dxdy - \int_{\{\underline{u}(0,\cdot)>\beta \}}e^{ay}f(\underline{u}(0,y))(\underline{u}(0,y)-\beta)dy =0.
\end{equation}
Recall that $| \{\underline{u}(0,y)\; : \; y \in \R\} \setminus [0,\beta]|>0$, and thus
\begin{equation}
\label{intNOnula}
 \int_{\{\underline{u}>\beta \}}e^{ay}|\nabla \underline{u}|^2dxdy>0.
\end{equation} 
Since $f(\underline{u}(0,y))(\underline{u}(0,y)-\beta)>0$ in $\{\underline{u}>\beta\}$
by \eqref{fposbeta}, 
\eqref{signolambda} and \eqref{intNOnula} lead to $1-2\lambda_a >0.$

Let  $\mu :=1-2\lambda_a>0$ and define
\begin{equation}
\label{SolucionFinal}
u(x,y):=
\underline{u}(\mu x,\mu y) \quad\text{ and } \quad c:=a(1-2\lambda_a)>0.
\end{equation}
We then have a solution pair
$(c,u)$ for  problem \eqref{problem}.

Note that $u \in H^1_{c}(\RR)$ since
$$\int_{\RR}e^{cy}\{|\nabla u|^2 + u^2\}dxdy =
 \int_{\RR}e^{a\overline{y}}\{|\nabla \underline{u}|^2 +\mu^{-2} \underline{u}^2\}d\overline{x}
 d\overline{y} < \infty $$
and $\underline{u} \in H^1_a (\RR).$

Finally, since $f\in C^{1,\gamma}$, the weak solution that we have found
can be shown to be classical, indeed $C^{2,\gamma}$ in all $\overline{\RR}$.
This is explained in the beginning of section~\ref{Limites}.
\end{proof}

\begin{rema}
\label{mirema}
{\rm
It is interesting to note the following
relation, already noted in \cite{H}, 
between the infimum value $I_a$ of our problem \eqref{mimin} and the speed $c$
of the traveling front. The formula, which is not strictly needed  anywhere else in this paper, provides however with an alternative proof of part (iv) of Theorem~\ref{Main}
on the comparison of the front speeds for different nonlinearities.

We claim that
\begin{equation}\label{formspeed}
c=a(1-2I_a)=a(1-2\lambda_a),
\end{equation}
where $a$ and $\lambda_a$ are the parameter and the multiplier in the proof of 
Proposition~\ref{PDE}.  
To show this formula we take a minimizing sequence $\{u_k\}$ made of $C^\infty$ functions with compact support, and we
test \eqref{lagrange} with $\phi=\partial_y u_k\in  H^1_a (\RR)$.
Integrating by parts in order to pass to the limit as $k\to\infty$, we obtain
\begin{equation*}
\begin{split}
0 & =  (1- 2\lambda_a)\int_{\RR}e^{ay}\partial_y \frac{|\nabla \underline{u}|^2}{2}  dxdy +
\int_{\partial \RR}e^{ay}\partial_y G(\underline{u}(0,y))dy
\\
& =-a\frac{1-2\lambda_a}{2}-a \int_{\partial \RR}e^{ay} G(\underline{u}(0,y))dy,
\end{split}
\end{equation*}
where in the last equality we have also integrated by parts.
We deduce that
$$
I_a=E_a(\underline{u})=\frac{1}{2}
\Gamma_a(\underline{u})-\frac{1-2\lambda_a}{2}
=\frac{1}{2}-\frac{1-2\lambda_a}{2}=\lambda_a,
$$
that together with \eqref{SolucionFinal} shows the claim.
}
\end{rema}

\section{Monotonicity}
\label{reordenamiento}
In this section  we show  that the front $u$ built in the previous section
can be taken to be  nonincreasing in the $y$ variable. This fact will be crucial
to show in next section that such nonincreasing front $u$ has  limits 1 and 0
as $y\to\mp\infty$. 

Note that it
suffices to show the existence of a nonincreasing minimizer, since the scaling
used in the proof of Proposition~\ref{PDE} does not change the monotonicity of the front.
As in \cite{H}, the existence of a nonincreasing minimizer will be a consequence of an 
inequality for monotone decreasing rearrangements in a new
variable $z$ defined by
$$
z=e^{ay}/a.
$$
 
\begin{propo}
\label{TMONOTONIA2D} 
The minimizer $\underline{u}$ of Corollary~{\rm \ref{exist2d}} can be taken 
to be nonincreasing in the $y$ variable.
\end{propo}

\begin{proof} 
We follow  ideas in \cite{H} and perform  
the change of variables $(x,z):= (x,e^{ay}/a),$  which takes $\RR$ into
 $(\R_+)^2= \left\lbrace (x,z)\, :\, x>0,z >0 \right\rbrace$, and the functionals $\Gamma_a$, 
 $E_a,$ into $\tilde{\Gamma}_a$, $\tilde{E}_a$, where
\begin{equation}\label{newfunct}
\tilde{\Gamma}_a(v):=\int\int_{(\R_+)^2}  \left\{|\partial_x v|^2+a^2z^2|\partial_z v|^2\right\} \, dxdz
\end{equation}
and
$$
\tilde{E}_a(v):= \frac{1}{2} \tilde{\Gamma}_a(v)
 + \int_0^{+\infty}G(v(0,z))\, dz.
$$

Let $\{u_k\}$ be the minimizing sequence for problem \eqref{problemconstraint}  given by Lemma~\ref{acotacionvariable}. The functions $\{u_k\}$ take values in $[0,1]$,
are continuous,
and have compact support in $\overline{\RR}$. Let $v_k$ be defined by
 $v_k(x,z):=u_k(x,y)$. Since $v_k$ is nonnegative, continuous, 
 and with compact support in
 $[0,+\infty)^2$, we may consider its monotone decreasing rearrangement in the $z$ variable,
 that we denote by $v_k^*$; see \cite{Ka}. That is, for each $x\geq 0$, we make the usual 
 one-dimensional monotone decreasing rearrangement of the function
 $v_k(x,\cdot)$ of $z>0$. Recall also that if we consider the even extension of $v_k$
 across $\{z=0\}$,  then $v_k^*$ coincides with the Steiner symmetrization
 of $v_k$ with respect to $\{z=0\}$.
 
As a consequence of equimeasurability, we have
$$\int_0^{+\infty}G(v_k^*(0,z))dz = \int_0^{+\infty}G(v_k(0,z))dz.$$
On the other hand, the inequality
$$  \tilde{\Gamma}_a(v_k^*) \leq \tilde{\Gamma}_a(v_k)$$
---and  thus $\tilde{E}_a(v_k^*)\leq \tilde{E}_a(v_k)$---
follows from a result of Landes \cite{L1} for monotone decreasing rearrangements
since the weight $w(x,z)=a^2z^2$
in \eqref{newfunct} is nonnegative and nondecreasing in $z\in (0,+\infty)$. 
It also follows from a previous
result of Brock \cite{Br} on Steiner symmetrization which requires 
$w$ to be nonnegative and $w^{1/2}(x,z)=a|z|$ to be even and convex.
These results require that the weight in front of $|\partial_x v|^2$ (which in our case
is identically one) does not depend on $z$.

Finally, we pull back the sequence $v_k^*$ to the $(x,y)$ variables and
name these functions $u_k^*$. We have that 
$$\Gamma_a(u_k^*)=\tilde{\Gamma}_a(v_k^*)
\leq\tilde{\Gamma}_a(v_k)=1$$ 
and
 $$E_a(u_k^*)=\tilde{E}_a(v_k^*)\leq \tilde{E}_a(v_k)= E_a(u_k). $$
 Let $\underline{u}^*$ be a weak limit in $H^1_a(\RR)$ of a subsequence of 
 $\{u_k^*\}$.  By the WLSC results of the previous section, it is easy to prove that we have necessarily $\underline{u}^* \in B_a$.  This is done exactly 
 as in the proof of Corollary~\ref{exist2d}. Thus, 
 $\underline{u}^*$ 
 is a minimizer which is nonincreasing in the $y$ variable. Note also that 
 it still takes values in $[0,1]$.
\end{proof}

\section{Limits at infinity}
\label{Limites} 

In this section we prove that the front $u$ for problem \eqref{problem} constructed in the previous
sections satisfies \begin{equation} \label{limitesverticales}
\lim_{y\to - \infty}u(x,y)=1\, \textrm{ and }\,
\lim_{y\to + \infty}u(x,y)=0 \qquad \textrm{ for all } x\ge 0,
\end{equation} 
and
\begin{equation}
\label{limitehorizontal} \lim_{x\to + \infty}u(x,y)=0\,\quad\textrm{ for all } y\in \R.
\end{equation} 
To establish \eqref{limitesverticales},  it will be crucial to use that $u$ is nonincreasing 
in the $y$ variable.

In what follows, we will be using the following regularity fact.
Assume that $u$ is a bounded $C^2$ 
function in $\R^2_+$, $C^1$ up to the boundary $\partial\R^2_+$, satisfying our
nonlinear problem
\begin{equation*} 
\begin{cases}
\Delta u+ cu_y = 0&\text{ in } \R^2_+\\ 
\dfrac{\partial u}{\partial\nu}= f(u)&\text{ on }\partial\R^2_+ .
\end{cases}
\end{equation*} 
Since $f$ is $C^{1,\gamma}$ for some $\gamma\in (0,1)$, we have that, for every $R>0$, 
$u\in C^{2,\gamma}(\overline{B^+_R})$ and
\begin{equation}
\label{maxpr2}
\norm{u}_{C^{2,\gamma}(\overline{B_R^+})} \leq C_R,
\end{equation}
for some constant
$C_R$ depending only on $c$, $\gamma$, $R$, and on upper bounds for
$\norm{u}_{L^\infty(B^+_{4R})}$ and $\norm{f}_{C^{1,\gamma}}$.
Here $B_R^+=\{(x,y)\in\R^2\, :\, x>0, |(x,y)|<R\}$.
This estimate is established by easily adapting the proof of \cite[Lemma 2.3(a)]{CSo}.
As a consequence of the estimate, we also deduce that
\begin{equation}
\label{gradbound}
|\nabla u|\in L^\infty (\RR).
\end{equation}

To establish  \eqref{limitesverticales}, we first need the following easy result
on limits as $|y|\to\infty$. It applies to any solution, not only to the variational
one constructed in previous sections.

\begin{lem}
 \label{lindelof}
Assume that $f(0)=f(1)=0$ and that $0\leq u \leq1$ 
is a solution of \eqref{problem}  satisfying $$\lim_{y\to - \infty}u(0,y)=1 \quad\text{and}\quad
\lim_{y\to + \infty}u(0,y)=0.$$

Then, for all $R>0$, we have
\begin{equation}
\label{maximumpr}
\lim_{y\to - \infty}u(x,y)=1, \quad  \lim_{y\to + \infty}u(x,y)=0,
\quad\text{and}\quad \lim_{|y|\to \infty}|\nabla u(x,y)|=0
\end{equation}
uniformly in $x \in [0,R]$.
\end{lem}

\begin{proof} 
For $t \in \R$ let us define $u^t(x,y):=u(x,y+t)$, also a solution of \eqref{problem}.  
We claim that
$$||u^{t}-1||_{L^{\infty}(B_{R}^{+})} +
||\nabla u^{t}||_{L^{\infty}(B_{R}^{+})} \rightarrow 0 \quad \textrm{ as } t 
\rightarrow -\infty.$$
Assume, by the contrary, that there exist $\varepsilon >0$ and $\{t_{k}\}\subset \R$
 with $t_{k} \rightarrow -\infty$ 
such that
\begin{equation}
\label{convunifAbsurdo}
||u^{t_{k}}-1||_{L^{\infty}(B_{R}^{+})} +
||\nabla u^{t_k}||_{L^{\infty}(B_{R}^{+})} \geq \varepsilon.
\end{equation}

Estimates \eqref{maxpr2} lead to the existence of a subsequence
 $\{t_{k_{j}}\}$ for which $u^{t_{k_{j}}}$ converges in
 $C^2(\overline{B_{R}^{+}})$ to $u^{\infty}$.
 By the hypothesis of the lemma we will have $0\leq u^{\infty}\leq 1$ and 
 \begin{equation*}
\begin{cases}
\Delta u^{\infty} + cu^{\infty} _y = 0&\text{ in } \R^2_+\\ 
u^{\infty} = 1 & \text{ on }\partial\R^2_+ \\
\dfrac{\partial u^{\infty} }{\partial\nu}= f(u^{\infty})=f(1)=0&\text{ on }\partial\R^2_+ .
\end{cases}
\end{equation*} 
From this and Hopf's boundary lemma, we deduce $u^{\infty}\equiv 1$ on $\RR$, 
which contradicts \eqref{convunifAbsurdo}. 

In an analogous way we can show the limits as $y \rightarrow +\infty$.
\end{proof}

We can now prove the existence of limits as $y\to\pm\infty$ for the variational
solution constructed in the last sections.

\begin{lem}
 \label{limitesverticalesdemo}
Let $f$ be  of positively-balanced bistable type 
or of combustion type as in Definition~{\rm \ref{nonlinearterm}}. Let $u$ be 
any front constructed as in Proposition~{\rm \ref{PDE}} from the nonincreasing 
minimizer of  Proposition~{\rm \ref{TMONOTONIA2D}}.  Then,
$$\lim_{y\to - \infty}u(x,y)=1\, \,\textrm{ and }
\lim_{y\to + \infty}u(x,y)=0 \,\, \,\textrm{ for all } x\ge 0.$$
\end{lem}

\begin{proof}
By Corollary~\ref{exist2d} we know that 
  $0\le u \le 1$, $u\in H^1_c(\RR)$, and that the set 
$\{u(0,y)\, : \,y\in \R\}$ is not contained in $[0,\beta]$.
We also know that  $u_y(x,y)\leq0$. Therefore, for all $x\ge 0$, 
there exist $L^-(x) \in (\beta,1]$ and 
$L^+(x) \in [0,1]$ such that 
$$
\lim_{y\to - \infty}u(x,y)=L^-(x)\quad \textrm{ and } \quad
\lim_{y\to + \infty}u(x,y)=L^+(x).
$$

Note that $L^+(x)\equiv 0$ is a  consequence of the inequalities of Lemma~\ref{ineq2}.

To prove that $L^-(x)\equiv 1$, by Lemma \ref{lindelof}
it is enough to show that $L^{-}(0)=1$. To do this, we consider 
the sequence of solutions $\{u^k\}$ 
defined by $u^k(x,y):=u(x,y+k)$. Then, as in the proof of Lemma \ref{lindelof}, 
as $k\to -\infty$ there exists a convergent 
subsequence to a solution $u^{\infty}$ of 
\begin{equation}
 \label{convunif2}
\begin{cases}
\Delta u^{\infty} + cu^{\infty} _y = 0&\text{ in } \R^2_+\\ 
u^{\infty} = L^{-}(0) & \text{ on }\partial\R^2_+ \\
\dfrac{\partial u^{\infty} }{\partial\nu}= f(L^{-}(0))&\text{ on }\partial\R^2_+ .
\end{cases}
\end{equation}
Since $u^{\infty}(x,y)=L^{-}(x)$, we have that $\partial_y u^{\infty}(x,y)=
\partial_{yy}u^{\infty}(x,y)\equiv 0$. Therefore, the first equation in \eqref{convunif2} leads to
 $\partial_{xx}u^{\infty}(x,y)=0$ 
for all $x >0$ and $y\in\R$. 
Since $u^{\infty}$ is bounded, then it must be constant equal to $L^{-}(0)$,
its value at $x=0$.

This and the last equation in \eqref{convunif2}  lead to 
$0=-\partial_{x}u^{\infty}(0,y)=f(L^{-}(0))$. Since $L^{-}(0)\in 
(\beta,1]$, and on this interval, $f$ vanishes only at 1 by hypothesis \eqref{fposbeta}, 
we conclude that $L^{-}(0)=1$.
\end{proof} 

It remains to prove \eqref{limitehorizontal} on the limits as $x\to +\infty$.
This  is  a simple consequence of the 
Harnack inequality and the fact that the variational solution lies in
$H^1_c(\RR)$.
 
 \begin{lem}
 \label{limitehorizontaldemo}
Let $f$ be  of positively-balanced bistable type
or of combustion type as in Definition~{\rm \ref{nonlinearterm}}.
Let $u$ be any front constructed as in Proposition~{\rm \ref{PDE}}. Then,
$$ \lim_{x\to + \infty}u(x,y)=0\,\quad\textrm{ for all } y\in \R.  $$
\end{lem}
 
\begin{proof}
Take any $y_0\in\R$. Since $u\in H^1_c(\RR)$ we have that
\begin{equation}
\label{integralcaja}
 \lim_{x_0\to + \infty} \int_{y_0 -1}^{y_0 +1}dy\int_{x_0 -1}^{+\infty}e^{cy} u^2 dx =0.
\end{equation}

Recall that $0\leq u\leq 1$ satisfies $\Delta u +cu_y=0$ in $\RR$. Thus, by the
Harnack inequality, for any $x_0>2$ we have 
$$
\sup_{B_{1}(x_0,y_0)}u\leq C\inf_{B_{1}(x_0,y_0)}u \leq 
C \int_{B_{1}(x_0,y_0)} u dxdy \leq C \left( \int_{B_{1}(x_0,y_0)} u^2 dxdy\right)^{1/2}$$
for different constants $C$ independent of $x_0$.
Using \eqref{integralcaja}, it follows that
$$\lim_{x_0\to + \infty}u(x_0,y_0)=0,$$
as claimed.
\end{proof}

\section{Uniqueness of speed and of solution with limits}
\label{unicidad}

In the first part of this section we establish a useful comparison principle,
Proposition~\ref{l3}, in the spirit of one in \cite{CSo}. It will lead first to the asymptotic
bounds on fronts stated in our main theorem (after building appropriate comparison
barriers in next section). Secondly, it will be used in the second part of this present section 
to establish a key result, Proposition~\ref{uniqueness} below.

Proposition~\ref{uniqueness} will have several important applications. First, the monotonicity
in $y$ of every solution with limits. Second, the uniqueness
of a speed and of a front with limits. And third, the comparison result between speeds 
corresponding to different ordered nonlinearities. The proof of the proposition
follows the powerful sliding method of Berestycki and Nirenberg~\cite{BN1}.

\subsection{A maximum principle}

We start with the following easy lemma.

\begin{lem}
\label{lemamio} 
Let $w$ be a $C^2$ function in $\R^2_+$, bounded by below, continuous up to $\partial \RR$,
and satisfying 
\begin{equation*} 
\Delta w + cw_y \leq  0  \quad\text{ in } \R^2_+ 
\end{equation*}
for some constant $c\in\R$. Assume also that $w\geq 0$ on $\partial \R^2_+$ and that, for every $R>0$, 
\begin{equation}
\label{liminf0}
\liminf_{|y|\to +\infty} w(x,y)\geq 0 \quad\text{ uniformly in $x\in[0,R]$}.  
\end{equation}  
Then, $w\geq 0$ in $\R^2_+$. 
\end{lem}

\begin{proof}
Consider the new function
$$
\overline w=\frac{w}{x+1} \quad\text{ for } x\geq 0,y\in \R.
$$
It satisfies
$$
\Delta \overline w + \frac{2}{x+1}\overline w_x+c\overline w_y\leq 0 \quad \textrm{ in } \RR.
$$

Let $\varepsilon >0$. Since $w$ is bounded below,  if $R$ is sufficiently large we have that 
\begin{equation}\label{ineqeps}
\overline w(x,y)\geq -\varepsilon
\end{equation}
for $x=R$. By assumption \eqref{liminf0},
we also have \eqref{ineqeps} for $x\in [0,R]$ and $|y|=S$ if 
$S$ is large enough (depending on $R$). 
Since,  \eqref{ineqeps} also holds
for $x=0$, the maximum principle applied in $(0,R)\times(-S,S)$ gives that 
$\overline w \geq -\varepsilon$ in $(0,R)\times(-S,S)$. 

Letting $S\to\infty$ we deduce that $\overline w \geq -\varepsilon$ in $(0,R)\times\R$. Now, letting  
$R\to\infty$ we conclude that $\overline w \geq -\varepsilon$ in $\RR$ for any $\varepsilon>0$.
Thus $\overline w \geq 0$  in $\RR$ and this finishes the proof.
\end{proof}

The following  maximum principle (in the spirit of one in \cite{CSo}) is a key ingredient
in the remaining of this section.  It will be applied to the difference of two solutions (and also of a supersolution and a solution) of our nonlinear problem.

\begin{propo}
\label{l3} 
Let $c\in\R$ and  $v$ be a $C^2$ bounded function in $\overline{\R^2_+}$
satisfying 
\begin{equation*} 
\Delta v + cv_y \leq 0\quad\text{ in } \R^2_+
\end{equation*} 
and that, for all $R>0$, 
\begin{equation}
\label{lim0}
\lim_{|y|\rightarrow\infty} v(x,y)= 0 \quad \mbox{  uniformly in } x\in [0,R]. 
\end{equation}
Finally, assume that there exists a nonempty set $H\subset\R$ such that
$v(0,y)> 0$ for $y\in H$, 
\begin{equation} \label{bdrycond}
\dfrac{\partial v}{\partial\nu}+ d(y)v\geq 0\quad\text{ if } y\not\in H
\end{equation}
and
\begin{equation}
\label{cotanewmann}
d(y)\ge 0 \quad\textrm{ if } y\not\in H,
\end{equation}
for some continuous function $d$ defined on $\R\setminus H$.
 
Then, $v>0$ in $\overline{\R^2_+}$. 
\end{propo}

\begin{proof}
We need to prove that $v\geq 0$ in $\overline{\R^2_+}$.  It then follows  that  $v>0$ in 
$\overline{\RR}$.
Indeed, since $H$ is nonempty, $v$ cannot be identically zero. If we assume that
$v=0$ at some point $(x_1,y_1)\in\overline{\RR}$, we obtain a contradiction using
the strong maximum principle (if $x_1>0$)
and using the Hopf's boundary lemma and \eqref{bdrycond} (if $x_1=0$,
since then $y_1\not\in H$ because $v(0,y_1)=0$).

Let
$$
A:=\inf_{\partial\RR}v.
$$
By \eqref{lim0} used with $x=0$, we have $A\leq 0$.  Thus, we can apply Lemma \ref{lemamio} to $w:=v-A$. We deduce
that $v\geq A$ in all of $\RR$. 

It only remains to prove that $A\geq 0$.
By contradiction, assume that $A<0$. Then, by its definition and since $v(0,y)\to 0$ as 
$|y|\to\infty$, we have that the infimum $A$ of $v$ on $\partial\RR$ is achieved at some point
$(0,y_0)$. Since we have proved that $v\geq A$ in all of $\RR$, then $(0,y_0)$ is also a 
minimum of $v$ in all $\overline\RR$. Since $v(0,y_0)=A<0$, $v$ is not identically constant, and thus
the Hopf's boundary lemma gives that
$-v_x(0,y_0)<0$. This is a contradiction with \eqref{bdrycond} and \eqref{cotanewmann} ---since 
$y_0\not\in H$ because $v(0,y_0)<0$.
\end{proof}

\subsection{Uniqueness}

The goal of this section is to establish uniqueness of the traveling speed, as well as
uniqueness ---up to
vertical translations--- of solutions to \eqref{problem}
which have limits 1 and 0, as $y\to\mp\infty$, on 
$\partial\R^2_+.$ 

We also prove in this section that every solution with the above limits is necessarily
decreasing in $y$.

All these three results will follow from the following proposition
---an analogue of Lemma 5.2 in \cite{CSo}.

\begin{propo}
\label{uniqueness}
Assume that $f$ satisfies \eqref{zeroes} and \eqref{noninc}, and let $c\in \R$. 
Let $u_1$ and $u_2$ be, respectively, a supersolution and a solution of \eqref{problem}
such that
$$
 0\leq u_i \le 1 \quad\text{and}\quad u_i(0,0)=1/2 
$$ 
for $i=1,2$. Assume that, for $i=1,2$ and all $R>0$,
\begin{equation}\label{bothlim}
\lim_{y\rightarrow -\infty}u_i(x,y)=1\, \mbox{ and }\, 
\lim_{y\rightarrow +\infty}u_i(x,y)=0 \quad\text{ uniformly in $x\in [0,R]$.}
\end{equation} 
For $t>0$, consider
$$
u_2^t(x,y):=u_2(x,y+t).
$$

Then, 
\begin{equation}\label{concl}
u_2^t\leq u_1\, \text{ in } \overline{\R^2_+}  \quad\text{for every } t>0.
\end{equation}  
As a consequence, $u_2\equiv u_1$ in $\overline{\R^2_+}$.

In addition, for any solution $u_2$ satisfying $0\leq u_2\leq 1$ and \eqref{bothlim},
we have  $\partial_y u_2<0$ in $\overline{\R^2_+}$.
\end{propo}

Note that if we apply the proposition to $u_1=u_2=u$, where $u$ is a solution to
\eqref{problem} taking values in $[0,1]$ and with
limits 1 and 0, the conclusion \eqref{concl} establishes that 
$u$ is nonincreasing in $y$. From this, the strong maximum principle and the
Hopf's boundary lemma applied to the linearized problem satisfied by $u_y$, 
we deduce that $u_y <0,$
as claimed in the last statement of the proposition.

Second,  by letting $t\to 0^+$ in  \eqref{concl} we deduce that $u_2\leq u_1$.
But $u_2$ is a solution and $u_1$ a supersolution, with $u_2(0,0)=u_1(0,0)$.
Again the strong maximum principle and the
Hopf's boundary lemma give that $u_2\equiv u_1$, as stated in the proposition.

The proposition also gives the uniqueness of a solution with limits 
(uniqueness up to vertical translations)
for a given speed $c$. For this, apply the proposition to two solutions after translating
them in the $y$ direction.

In the proof of our main theorem in the last section, we will give two other important
applications of the proposition. First, the uniqueness of a speed admitting a solution
with limits. This will follow from the fact that any front $u_1$ with limits 1 and 0
is necessarily decreasing, and hence
the terms $c_1\partial_yu_1$ and $c_2 \partial_yu_1$ 
will be comparable for two different speeds. Since one of the functions in the proposition
may be taken to be only a supersolution, this will lead to the uniqueness of speed.

A similar argument will show the comparison of speeds corresponding
to two different ordered nonlinearities.

\begin{proof}[Proof of Proposition \ref{uniqueness}]
As explained above, we only need to prove \eqref{concl}. The subsequent 
statements follow easily from this.

Note first that $u_i$ are not identically constant,
by the assumption in \eqref{bothlim} about their limits
as $y\rightarrow\pm\infty$.
Therefore, since $0\le u_i \le 1$ and $f(0)=f(1)=0$, the strong maximum principle leads to
$0<u_i<1$ for $i=1,2$. 
 
Let $\delta>0$ be the constant in assumption \eqref{noninc} for $f$.
By hypothesis \eqref{bothlim}, there exists a compact
interval $[a,b]$ in $\R$ such that, for $i=1,2$,
\begin{eqnarray*}
u_i(0,y)\in (1-\delta, 1) \quad & \mbox{if } y\le a, &  \mbox{ and}\\
u_i(0,y)\in (0,\delta) \quad  & \mbox{if } y\ge b. & \
\end{eqnarray*}

Note that $u_2^t$ is also a solution of \eqref{problem}, and hence
\begin{equation*}
\begin{cases}
\Delta (u_1-u_2^t)+ c(u_1-u_2^t)_y\leq 0 & \mbox{ in }\R^2_+ \\
-(u_1-u_2^t)_x\geq -d^t(y)(u_1-u_2^t)  & \mbox{ on }\partial\R^2_+,
\end{cases}
\end{equation*}
where 
$$
d^t(y)=-\,\dfrac{f(u_1)-f(u_2^t)}{u_1-u_2^t}(0,y)
$$
if $(u_1-u_2^t)(0,y)\not =0$, and $d^t(y)=-f'(u_1(0,y))=-f'(u_2^t(0,y))$ otherwise.
Note that $d^t$ is a continuous function since $f$ is $C^1$.

Note that we also have, for all $R>0$,
$$
\lim_{|y|\rightarrow \infty}
(u_1-u_2^t)(x,y)=0 \quad\text{ uniformly in } x\in[0,R]. 
$$

We finish the proof in three steps.

\medskip

{\it Step} 1. We claim that $u_2^t<u_1$ in $\overline{\RR}$ for $t>0$ large enough.

To show this, we take $t>0$ sufficiently large such that
$u_2^t (0,y)< u_1(0,y)$ for $y\in [a,b]$. 
This is possible since
$u_2(0,y+t)\rightarrow 0$ as $t\rightarrow +\infty$ and $u_1>0$.
We apply Proposition~\ref{l3} to $v=u_1-u_2^t$, with 
$$
H= (a,b) \cup \{y\in\R : (u_1-u_2^t)(0,y)>0\} .
$$
Clearly, $v(0,y)>0$ in $H$. 

To show that $d^t\ge 0$ in $\R\setminus H$, let $y\not\in H$. 
There are two possibilities. First, if $y\ge b$ then $y+t\ge b$ also.
Therefore, $u_1(0,y)\leq \delta$ and $u_2^t(0,y)=u_2(0,y+t)\leq\delta$.
We conclude that $d^t(y)\ge 0$, since $f'\leq 0$ in $(0,\delta)$ by~\eqref{noninc}. 

The other possibility is that $y\le a$. 
In this case, we have $u_1(0,y)\geq 1-\delta$, and since $y\not\in H$
then $(u_1-u_2^t)(0,y)\le 0$. Therefore 
$u_2^t(0,y)\geq u_1(0,y)\geq 1-\delta$, and we conclude 
$d^t(y)\ge 0$, again by~\eqref{noninc}.

Proposition~\ref{l3} gives that $u_1-u_2^t>0$ in $\overline{\R^2_+}$.

\medskip

{\it Step $2$. Claim}: If $u_2^t\le u_1$ for some $t>0$,
then $u_2^{t+\mu}\le u_1$ for every $\mu$ small enough
(with $\mu$ either positive or negative).

This statement will finish the proof of the lemma, since then
$\{t>0 : u_2^t\le u_1\}$ is a nonempty, closed and open set in 
$(0,\infty)$, and hence equal to this interval. We conclude
$u_2^t\le u_1$ for all $t>0$.

To prove the claim of Step 2, we first show that
\begin{equation}
\label{notiden}
\mbox{if } t>0 \mbox{ and } u_2^t\le u_1, \mbox{ then } u_2^t\not\equiv u_1 .
\end{equation}

Once \eqref{notiden} is known, we can finish the proof of the claim
as follows.
First, by the strong maximum principle and Hopf's boundary lemma, 
$u_2^t< u_1$ in $\overline{\R^2_+}$.
Let $K_t$ be a compact interval such that,
on $\R\setminus K_t$, both $u_1$ and $u_2^t$ take values in $(0,\delta/2)\cup (1-\delta/2,1)$.
Recall that $(u_1-u_2^t)(0,\cdot)>0$ in the compact set $K_t$.
By continuity and the existence of limits at infinity,
we have that if $|\mu|$ is small enough, then
$(u_1-u_2^{t+\mu})(0,y)>0$ for $y\in K_t$ and 
$u_2^{t+\mu}(0,y)$ takes values in $(0,\delta)\cup (1-\delta,1)$  for $y\not\in K_t$.
Hence, we can apply Proposition~\ref{l3} to $v=u_1-u_2^{t+\mu}$ with
$H=K_t$, since $d^{t+\mu}\ge 0$ outside $K_t$. 
We conclude $u_1-u_2^{t+\mu}>0$ in $\overline{\RR}$.

\medskip

{\it Step} 3. Here we establish \eqref{notiden},
therefore completing the proof of Step~2 and of the proposition.
We assume that $t>0$ and 
$u_2^t\le u_1$, and we need to show that $u_2^t\not\equiv u_1$.

To prove this, consider first the case when both functions in
the lemma are the same, that is, $u_1\equiv u_2$.
Assume that $t>0$ and $u_2^t\equiv u_1\equiv u_2$. Then, the function 
$u_2(0,y)$ is $t$-periodic. But this is a contradiction with 
the hypothesis  \eqref{bothlim} on limits.
Therefore, in the case $u_1\equiv u_2$, the two steps above can be
carried out. We conclude that, for every solution $u_2$ as in
the lemma, we have $u_2^t\le u_2$ for every $t>0$. In particular,
$\partial_y u_2\leq 0$ and, by the strong maximum principle and Hopf's boundary lemma,
$\partial_y u_2<0$.

Finally, consider the general
case of a supersolution $u_1$ and a solution $u_2$.
Assume that $t>0$ and $u_2^t\equiv u_1$.
Then $1/2=u_1(0,0)=u_2^t(0,0)=u_2(0,t)$.
Moreover, $u_2(0,0)=1/2$ by hypothesis. Hence, $u_2(0,0)=u_2(0,t)$. 
This is a contradiction, since in the previous paragraph 
we have established that $u_2$ is decreasing in $y$.
\end{proof}

\section{Explicit traveling fronts}
\label{explicit}

In this section we construct an explicit supersolution of the linearized problem
for \eqref{problem} in the case of positively-balanced bistable nonlinearities
satisfying 
\begin{equation}\label{dernegsec}
f'(0)<0 \quad\text{ and }\quad f'(1)<0.
\end{equation}
It will lead to our result on the asymptotic behavior
of traveling fronts. In addition, we construct a family of explicit traveling fronts corresponding
to some positively-balanced bistable nonlinearities satisfying \eqref{dernegsec}.

To simplify the notation in this section,
by rescaling the independent variables we may assume that the speed of the front is 
$$
c=2.
$$

Recall that our nonlinear problem, when written for the trace $v=v(y)$ of functions on $x=0$,
becomes \eqref{fracpb} with $c=2$, i.e.,
$$
(-\partial_{yy} -2\partial_y)^{1/2} v = f(v) \quad\text{in } \R.
$$ 
As in \cite{CSi}, the construction of explicit fronts will be based on the fundamental solution 
for the homogeneous heat equation associated to the previous fractional operator 
in~$\R$, that is, equation 
\begin{equation}
\label{fracheat}
\partial_t v + (-\partial_{yy} -2\partial_y)^{1/2} v=0
\end{equation}
for functions $v=v(y,t)$.
Taking one more derivative $\partial_t$ in \eqref{fracheat}, we see that 
the solution of this problem at time $t$ (given an initial condition $v_0$)
coincides with the value of $w(x=t,\cdot)$ for the solution of
\begin{equation}
\label{extpb}
\begin{cases}
L_2w:=\Delta w + 2w_y=0&\text{ in } \R^2_+\\
w(0,\cdot)=w_0 &\text{ on } \R ,
\end{cases}
\end{equation}
where the operator $L_2$ acts on functions $w=w(x,y)$.
Thus, the heat kernel for \eqref{fracheat} coincides with the Poisson kernel for \eqref{extpb}.

To compute such Poisson kernel, as in \cite{CMS} we start with the observation that if
$$
w=e^{-y}\phi,
$$
then
\begin{equation}\label{equivopr}
L_2 w= \Delta w + 2w_y=0 \qquad\text{ if and only if }\qquad -\Delta \phi + \phi = 0.
\end{equation}
The fundamental solution of Helmholtz's equation, solution of
$$ -\Delta \Phi +  \Phi = \delta_0,$$
is given by
$$ \Phi (r) = \frac{1}{2\pi} K_0(r)$$
where $r=\sqrt{x^2+y^2}$ and $K_0$ is the modified Bessel function of the second kind
with index $\nu=0$ (see \cite{AS}).
The function $K_0=K_0(s)$ is a positive and decreasing function of $s>0$,
whose asymptotic behavior at $s=0$ is given by
$$
K_0(s)=-\log s+ \textrm{o}(|\log s|) \quad \text{ as } s\to 0.
$$
For $s\to +\infty$, all  modified Bessel functions of the second kind $K_\nu$ have the
same behavior
\begin{equation}
\label{asyk0}
K_\nu(s) =\sqrt{\pi/2} \,\, s^{-1/2}e^{-s}+ \textrm{o}(s^{-1/2}e^{-s}) 
\quad \mbox{ as $s\to +\infty$.}
\end{equation}

By considering the fundamental solution $\Phi$ but now with pole at a point $(x_0,y_0)\in
\RR$, subtracting from it $\Phi$ with pole at the reflected point $(-x_0,y_0)$, and applying
the divergence theorem, one sees that the Poisson kernel for the Helmholtz's equation
$-\Delta\phi +\phi=0$ in the half-plane $\RR$ is given by 
$$-2\Phi_x=-\frac{1}{\pi}\frac{x}{r}K_0'(r).$$ 
Writing this convolution formula for $w=e^{-y}\phi$, we deduce that
the Poisson kernel for \eqref{extpb} is given by
$$-2e^{-y}\Phi_x=-e^{-y}\frac{x}{\pi r}K_0'(r).$$ 

To avoid its singularity at the origin, given any constant $t>0$,
we consider the Poisson kernel after ``time'' $x=t$ and define
\begin{equation}\label{defPt}
\begin{split}
P^t(x,y) & :=-2G^t_x =-e^{-y}\dfrac{x+t}{\pi\sqrt{(x+t)^2+y^2}}\, K_0'(\sqrt{(x+t)^2+y^2})\\
& =e^{-y}\dfrac{x+t}{\pi\sqrt{(x+t)^2+y^2}} \,K_1(\sqrt{(x+t)^2+y^2})
\end{split}
\end{equation}
where $K_1=-K_0'$ is the modified Bessel function of the second kind with index $\nu=1$,
and where
\begin{equation*}
G^t(x,y) : =\frac{1}{2\pi} e^{-y}\, K_0(\sqrt{(x+t)^2+y^2}).
\end{equation*}
By \eqref{equivopr}, $G^t$ is a solution of the homogeneous equation $L_2w=0$ in $\RR$.  
Thus, so it is $P^t= -2G^t_x$.

Finally, the explicit traveling front will be given by 
\begin{equation*}
\begin{split}
u^t(x,y) & :=\int_y^{+\infty} P^t(x,z) dz \\
&= \int_y^{+\infty} e^{-z}\dfrac{x+t}{\pi\sqrt{(x+t)^2+z^2}} 
\,\, K_1(\sqrt{(x+t)^2+z^2}) dz. 
\end{split}
\end{equation*}

Next, let us check all the properties of $P^t$, for $t>0$, that will be needed in order to use
it as a supersolution of the linearized problem for \eqref{problem}. 
We know that $L_2P^t=0$ in~$\RR$. Using \eqref{asyk0} we see that $P_t$ is
bounded in all $\RR$.
We also have that $P^t>0$ in $\overline{\RR}$ since $K_0$ is radially
decreasing. Next, we have that, for every $R>0$, $P^t(x,y)\to 0$ as $|y|\to\infty$
uniformly in $x\in [0,R]$. This follows from the last equality in \eqref{defPt} and 
from \eqref{asyk0}. Also from the last equality in \eqref{defPt} we see that
\begin{equation}\label{derPt}
\frac{-\partial_xP^t}{P^t}(0,y) =
\frac{-1}{t}\left\{
1-\frac{t^2}{t^2+y^2}+ \frac{t^2}{\sqrt{t^2+y^2}} \,\, 
\frac{K_1'(\sqrt{t^2+y^2})}{K_1(\sqrt{t^2+y^2})}\right\} .
\end{equation}
Now, using that $K_1'=(-1/2)(K_0+K_2)$ and 
the asymptotic behavior \eqref{asyk0}, we deduce that
\begin{equation}\label{limderPt}
\lim_{|y|\to\infty} \frac{-\partial_xP^t}{P^t}(0,y) = \frac{-1}{t}.
\end{equation}
Finally, since
$$
P^t(0,y) =e^{-y}\dfrac{t}{\pi\sqrt{t^2+y^2}} \,K_1(\sqrt{t^2+y^2})
$$
and $K_1$ has the asymptotic behavior \eqref{asyk0}, we deduce
\begin{eqnarray}
\label{as1}
P^t(0,y) &= &\frac{t}{\sqrt{2\pi}} \,\, y^{-3/2}e^{-2y}+ \textrm{o}(y^{-3/2}e^{-2y}) 
\quad \mbox{ as $y\to +\infty$, \,\, and}\\
\label{as2}
P^t(0,y) &=& \frac{t}{\sqrt{2\pi}} \,\, (-y)^{-3/2}+ \textrm{o}((-y)^{-3/2}) 
\quad \mbox{ as $y\to -\infty$.}
\end{eqnarray}

We can now establish our result on explicit traveling fronts.
We need to verify that each $u^t$ is a traveling front for some nonlinearity $f^t$
 of positively-balanced bistable type
 satisfying $(f^t)'(0)<0$ and $(f^t)'(1)<0$. 
 
\begin{proof}[Proof of Theorem \ref{theoexpl}]
The statements for $u^{t,c}$ follow from the corresponding ones for $u^{t,2}=u^t$.
To prove them for $u^t$, note first that the solution of \eqref{extpb}
when $w_0\equiv 1$ is $w\equiv 1$. We deduce that its Poisson kernel satisfies
$$
\int_{-\infty}^{+\infty} P^t(x,z) dz =1
$$
for all $x>0$. It follows that $0<u^t<1$ and that $\lim_{y\to -\infty}u^t(x,y)=1$ for all $x\geq 0$. 
Clearly we also have $\lim_{y\to +\infty}u^t(x,y)=0$. In addition, $\partial_yu^t=
-P^t<0$ in $\overline\RR$. 

Next, let us see that we have $L_2 u^t =0$ in $\RR$.  Indeed, $\partial_y L_2 u^t=
L_2 u^t_y=-L_2 P^t =0$. Thus, $L_2 u^t$ is a function of $x$ alone. But
\begin{equation*}
\begin{split}
L_2 u^t  = & \int_y^{+\infty} \partial_{xx}\left\{
e^{-z}\dfrac{x+t}{\pi\sqrt{(x+t)^2+z^2}} \, 
K_1( \sqrt{(x+t)^2+z^2})\right\} dz\\
& -\partial_y\left\{e^{-y}\dfrac{x+t}{\pi\sqrt{(x+t)^2+y^2}} \, K_1(\sqrt{(x+t)^2+y^2})
\right\} \\
& -2 e^{-y}\dfrac{x+t}{\pi\sqrt{(x+t)^2+y^2}} \, K_1(\sqrt{(x+t)^2+y^2}).
\end{split}
\end{equation*}
Since $L_2u^t$ does not depend on $y$, we may let $y\to +\infty$ in this expression. 
Now, using that $K_j'=(-1/2)(K_{j-1}+K_{j+1})$ for all $j$ and 
that all functions $K_\nu(s)$ have the asymptotic behavior \eqref{asyk0}, 
we deduce that $L_2u^t\equiv 0$ in $\RR$.

The asymptotic behaviors for $-u^t_y=P^t$ in the statement of the theorem follow from  
\eqref{as1} and \eqref{as2}.

Next, we find the expression for the nonlinearity $f^t$. Since $u^t(0,\cdot)$ is decreasing
from 1 to 0, $f^t$ is implicitly well defined in $[0,1]$ by
\begin{equation}\label{formf}
\begin{split}
f^t(u^t(0,y))  &:= -u^t_x(0,y) = \int_y^{+\infty} 2G^t_{xx}(0,z) dz\\
& =  \int_y^{+\infty} 2( - G^t_{yy}-2 G^t_{y})(0,z) dz = 2(G^t_y+2G^t)(0,y)\\
& = \frac{1}{\pi} e^{-y}\left\{ K_0(\sqrt{t^2+y^2})
- \frac{y}{\sqrt{t^2+y^2}} K_1(\sqrt{t^2+y^2}) \right\}.
\end{split}
\end{equation}
From this, we clearly see that $f^t(0)=f^t(1)=0$, again by \eqref{asyk0}.

The rest of properties of $f^t$ will be deduced from the following implicit formula
for its derivative. We have $(f^t)'(u^t(0,y)) u^t_y(0,y) = -u^t_{xy}(0,y)$ and thus, by \eqref{derPt},
\begin{equation*} 
\begin{split}
(f^t)'&(u^t(0,y)) = \frac{-u^t_{xy}}{u^t_y}(0,y) \\
&=\frac{-\partial_xP^t}{P^t}(0,y) =
\frac{-1}{t}\left\{
1-\frac{t^2}{t^2+y^2}+ \frac{t^2}{\sqrt{t^2+y^2}} \,\, 
\frac{K_1'(\sqrt{t^2+y^2})}{K_1(\sqrt{t^2+y^2})}\right\} \\
&= \frac{-1}{t}\left\{
1-\frac{t^2}{t^2+y^2}- \frac{t^2}{\sqrt{t^2+y^2}} \,\, 
\frac{K_0(\sqrt{t^2+y^2})+K_2(\sqrt{t^2+y^2}) }{2K_1(\sqrt{t^2+y^2})}\right\}  \\
&= \frac{t}{\sqrt{t^2+y^2}}\left\{
-\frac{\sqrt{t^2+y^2}}{t^2}+ \frac{1}{\sqrt{t^2+y^2}}  
+\frac{K_0(\sqrt{t^2+y^2})+K_2(\sqrt{t^2+y^2}) }{2K_1(\sqrt{t^2+y^2})}\right\} \\
&=: \frac{t}{\sqrt{t^2+y^2}}\,\, h^t(y).
\end{split}
\end{equation*}
It turns out that the function $\{K_0(s)+K_2(s)\}/(2K_1(s))$ is a decreasing function of
$s\in (0,+\infty)$ which behaves as $1/s$ at $s=0^+$ and as $1+1/(2s)$ at $+\infty$.
Therefore,
\begin{equation*}
(f^t)'(0)=(f^t)'(1)=-\frac{1}{t} <0.
\end{equation*}
It also follows that $h^t$ is a decreasing function of $y\in [0,+\infty)$ satisfying $h^t(0)
=(K_0(t)+K_2(t))/(2K_1(t))>0$
and $\lim_{y\to +\infty}h^t(y)=-\infty$. Therefore, since $h^t$ is an even function of $y$,
there exists a $y^t>0$ such that $h^t$ is negative in $(-\infty,-y^t)$,
positive in $(-y^t,y^t)$, and negative in $(y^t,+\infty)$. As a consequence,
for some $0<\gamma_1<\gamma_2<1$, we have that
$(f^t)'$ is negative in $(0,\gamma_1)$,
positive in $(\gamma_1,\gamma_2)$, and negative in $(\gamma_2,1)$. 
This gives that $f^t$ has a unique zero in $(0,1)$ and that $f^t$ is of bistable type.

We finally check the positively-balanced character of $f^t$ (after the end of the
proof we give an alternative,
more synthetic argument for this).  Using formula \eqref{formf} for $f^t$ in terms of 
$G^t$ and that $G^t_{xx}=-\partial_y(G^t_y+2G^t)$, we have
\begin{equation*}
\begin{split}
\int_0^1 f^t(s)ds &=\int_{\R} f(u^t(0,y)) (-u^t_y)(0,y) dy   =
-4\int_{\R} \{(G^t_y+2G^t)G^t_x\}(0,y) dy\\
& = 4\int_{\RR} \{  (G^t_y+2G^t)G^t_{xx} + (G^t_{xy}+2G^t_x)G^t_x\} dx dy\\
& = 4\int_{\RR} \{  -\partial_y(G^t_y+2G^t)^2/2 + \partial_y (G^t_{x})^2/2+2(G^t_x)^2\} dx dy\\
& = 8\int_{\RR} (G^t_x)^2 dx dy >0,
\end{split}
\end{equation*}
which finishes the proof.
\end{proof}

The following is an alternative way to prove that the integral of $f^t$ is positive.
It is more synthetic but it relies on deeper results. Assume by contradiction that
$\int_0^1f^t(s)ds\leq 0$. Then, by the remarks made after Theorem~\ref{Main}, 
the existence part of
Theorem~\ref{Main}, and the results of \cite{CSo}, there exists a solution of \eqref{problem} 
for some $c\leq 0$ and which satisfies the limits \eqref{limits0}. But $u^t$ is also a
solution of \eqref{problem}, now with  $c=2$, and satisfying the limits \eqref{limits0}. 
By the uniqueness of the speed proved in Theorem~\ref{Main}, we arrive to a 
contradiction.

\section{Proof of Theorem \ref{Main}}
\label{pruebafinal}

In this section we use all previous results to establish our main theorem.

\begin{proof}[Proof of Theorem \ref{Main}]

Let $f$ be  of positively-balanced bistable type
or of combustion type as in Definition~{\rm \ref{nonlinearterm}}. 

\medskip

{\em Part $($i$)$}. This first part has been established in Proposition~{\rm \ref{PDE}}, together with
Lemma~\ref{limitesverticalesdemo} 
where we proved the existence of limits as $y\to\pm\infty$.
That $0<u<1$ follows from the strong maximum principle and the Hopf's boundary lemma,
since we know that $0\leq u\leq 1$ and $f(0)=f(1)=0$.

\medskip

{\em Part $($ii$)$}. 
Let $(c_1,u_1)$ and $(c_2,u_2)$ be two solutions pairs
with $u_i$ taking values in $[0,1]$ and having limits 1 and 0 as $y\to\mp\infty$.
By Lemma~\ref{lindelof}, both $u_i$ satisfy the uniform limits assumption \eqref{bothlim}.
Translate each one of them in the $y$ variable so that both satisfy $u_i(0,0)=1/2$.
Assume that $c_1\leq c_2$. Proposition~\ref{uniqueness} applied with $c=c_1$
gives that the solution $u_1$ is decreasing in $y$. Thus
$$
0=\Delta u_1+c_1\partial_y u_1 \geq \Delta u_1+c_2 \partial_y u_1,
$$
and hence $u_1$ is a supersolution for the problem with $c=c_2$.
Proposition~\ref{uniqueness} applied with $c=c_2$
leads to $u_1\equiv u_2$. As a consequence,
since $\partial_y u_1 <0$, we deduce from the equations that $c_1=c_2$.

\medskip

{\em Part $($iii$)$}. 
The monotonicity in $y$ of a variational solution as in part (i) was established
in Proposition~\ref{TMONOTONIA2D}. From $u_y\leq 0$ we deduce $u_y<0$ using
the  strong maximum principle and the Hopf's boundary lemma for the
linearized problem satisfied by $u_y$.
The existence of its vertical and horizontal limits has been proved in 
Lemmas~\ref{limitesverticalesdemo}
and \ref{limitehorizontaldemo}. 
Let us now show that $u_x\leq 0$ in case that $f$ is of combustion type ---this fact
is not true for bistable nonlinearities since the normal derivative $-u_x=f(u)$ changes
sign on $\{x=0\}$. Note that $u_x$ is a solution of $(\Delta+c\partial_y)u_x=0$ in $\RR$,
it is bounded by \eqref{gradbound}, and has limits 0 as $|y|\to\infty$ uniformly on compact sets
of $x$ by \eqref{maximumpr}. In addition, $u_x=-f(u)\leq 0$ on $x=0$. 
Lemma~\ref{lemamio} leads to $u_x\leq 0$ in $\RR$.

\medskip

{\em Part $($iv$)$}. 
We can give two different proofs of  this part. 
Let $f_1\geq f_2$, not being identically equal. Let $(c_i,u_i)$ be the unique solution pair
for the nonlinearity $f_i$ 
with $u_i$ taking values in $[0,1]$ and having limits as $y\to\pm\infty$.

The first proof is variational and uses formula \eqref{formspeed} for the speed.
Take $a>0$ small enough so that both problems \eqref{problemconstraint}, for 
$f_1$ and for $f_2$,
can be minimized in $B_a$. Since $G_1\leq G_2$, the minimum values
satisfy $I_{1,a}\leq I_{2,a}$, in fact with a strict inequality since $G_1\not\equiv G_2$
and the minimizers $\underline{u}_1$ and $\underline{u}_2$
 take all the values in $(0,1)$. Thus, from  \eqref{formspeed}
we deduce $c_1>c_2$.

The second proof of (iv) is non-variational. Recall that
by Lemma~\ref{lindelof}, both $u_i$ satisfy the uniform limits assumption \eqref{bothlim}.
Translate each front in the $y$ variable so that both satisfy $u_i(0,0)=1/2$.
Assume, arguing by contradiction, that $c_1\leq c_2$. 
Since $u_1$ is decreasing in $y$, we have
$$
0=\Delta u_1+c_1\partial_y u_1 \geq \Delta u_1+c_2 \partial_y u_1.
$$
In addition,
$$
\frac{\partial u_1}{\partial\nu}=f_1(u_1)\geq f_2(u_1) \quad\text{on }\partial\RR.
$$
Hence, $u_1$ is a supersolution for the problem with $c=c_2$ and $f=f_2$.
Proposition~\ref{uniqueness} leads to $u_1\equiv u_2$. As a consequence,
we obtain $f_1\equiv f_2$ ---since $u_1$ takes all the values of $(0,1)$.
This is a contradiction.

\medskip

{\em Part $($v$)$}. To establish this part, 
it suffices to show the bounds for $-u_y$. From them, the ones
for $u$ and $1-u$ follow by integration. Defining $\tilde u$ by $u(x,y)=\tilde u (cx/2,cy/2)$,
we see that $\tilde u$ is a front with speed 2 for problem \eqref{problem}
with nonlinearity $(2/c)f$. Since the constants on the 
bounds of part (v) do not reflect the dependence
on $f$, we may rename $\tilde u$ by $u$, and $(2/c)f$ by $f$,
and assume that $u$ is a front for the nonlinearity $f$ with speed
$c=2$. Note however that the factor $e^{-2y}$ in \eqref{asyderpos}
will change to  $e^{-cy}$ in \eqref{asyderpos} after the scaling.

As stated in the theorem, 
the lower bounds for $-u_y$ hold for any $f$ of positively-balanced bistable type
or of combustion type. To prove them, we take $t>0$ small enough such that
$$
-\frac{1}{2t}\leq \min_{[0,1]} f'. 
$$
For such $t$, consider the Poisson kernel $P^t$ defined by \eqref{defPt}
and, for any positive constant $C>0$, the function
$$
v:=C(-u_y)-P^t .
$$
Note that $\Delta v+2v_y=0$ in $\RR$.
Using \eqref{limderPt} and that $-u_y$ and $P^t$ are positive, we infer
$\{-\partial_{x} v+ (2t)^{-1} v\}(0,y) \geq 0$  
for $|y|$ large enough, say for $y$ in the complement of a compact interval $H$.
Next, take the constant $C>0$ large enough such that $v >0$ in the compact set $H$.
By \eqref{gradbound}, the limits of $u_y$ established in \eqref{maximumpr}, and the 
properties of $P^t$ checked in section~\ref{explicit}, we can apply 
Proposition~\ref{l3} with $c=2$ and $d(y)=(2t)^{-1}$ to deduce that $v>0$ in $\RR$.
By using the asymptotic behaviors \eqref{as1} and  \eqref{as2}  of $P^t$ at $\pm\infty$,
we conclude the two lower bounds for $-u_y$.

To prove the upper bounds for $-u_y$ we need to assume that $f'(0)<0$ and $f'(1)<0$.
We proceed in the same way as for the lower bounds, but
replacing the roles of $-u_y$ and $P^t$. We now take $t>0$ large enough
such that
$$
\max\{f'(0),f'(1)\}< -\frac{2}{t}.
$$ 
Using \eqref{limderPt}, for any $C>0$,
$\tilde v:=CP^t-(-u_y)$ satisfies $\{ -\partial_{x} \tilde v+ (2/t) \tilde v\} (0,y) \geq 0$  
for $|y|$ large enough, say for $y$ in the complement of a compact interval $H$. 
One proceeds exactly as before to obtain
$\tilde v>0$ in $\RR$ for $C$ large enough.
This gives the desired upper bounds for $-u_y$.
\end{proof}

\end{document}